\documentclass[a4paper]{amsart}

\title[{Ordinary modules for vertex algebras of $\mathfrak{osp}_{1|2n}$}]{Ordinary modules for vertex algebras of $\mathfrak{osp}_{1|2n}$}

\author{Thomas Creutzig}
\address{Department Mathematik, FAU Erlangen, Cauerstraße 11
91058 Erlangen, Germany}
\email{creutzig@ualberta.ca}

\author{Naoki Genra}
\address{Kavli Institute for the Physics and Mathematics of the Universe (WPI), The University of Tokyo Institutes for Advanced Study, The University of Tokyo, Kashiwa, Chiba 277-8583, Japan}
\email{naoki.genra@ipmu.jp}

\author{Andrew Linshaw}
\address{Department of Mathematics, University of Denver, Denver, CO 80210}
\email{andrew.linshaw@du.edu}

\usepackage{cite}
\usepackage{latexsym}
\usepackage{amsmath}
\usepackage{amsfonts}
\usepackage{amssymb}
\usepackage{amsxtra}
\usepackage{mathrsfs}

\newtheorem{definition}{Definition}[section]
\newtheorem{proposition}[definition]{Proposition}
\newtheorem{theorem}[definition]{Theorem}
\newtheorem{corollary}[definition]{Corollary}
\newtheorem{lemma}[definition]{Lemma}
\newtheorem{remark}[definition]{Remark}

\numberwithin{equation}{section}

\def\WL{{\mathbf L}}
\def\AL{{ L}}

\newcommand{\Z}{\mathbb{Z}}
\newcommand{\Q}{\mathbb{Q}}
\newcommand{\R}{\mathbb{R}}
\newcommand{\C}{\mathbb{C}}
\newcommand{\cC}{\mathcal{C}}
\newcommand{\cD}{\mathcal{D}}
\newcommand{\cO}{\mathcal{M}}
\newcommand{\cF}{\mathcal{F}}
\newcommand{\cG}{\mathcal{G}}
\newcommand{\cH}{\mathcal{H}}
\newcommand{\cCloc}{\mathcal{C}_A^{\text{loc}}}
\newcommand{\cCtw}{\mathcal{C}_A^{\text{tw}}}
\newcommand{\cDloc}{\mathcal{D}_A^{\text{loc}}}
\newcommand{\W}{\mathcal{W}}
\newcommand{\g}{\mathfrak{g}}
\newcommand{\KL}{\text{KL}}
\newcommand{\CDO}{\mathcal{D}^{\text{ch}}}

\newcommand{\Hom}{\operatorname{Hom}}
\newcommand{\Ker}{\operatorname{Ker}}
\newcommand{\Span}{\operatorname{Span}}

\begin{document}
\maketitle

\begin{abstract}
We show that the affine vertex superalgebra $V^k(\mathfrak{osp}_{1|2n})$ at generic level $k$ embeds in the equivariant $\mathcal W$-algebra of $\mathfrak{sp}_{2n}$ times $4n$ free fermions. This has two corollaries: (1) it provides a new proof that for generic $k$, the coset $\text{Com}(V^k(\mathfrak{sp}_{2n}), V^k(\mathfrak{osp}_{1|2n}))$ is isomorphic to $\mathcal W^\ell(\mathfrak{sp}_{2n})$ for $\ell = -(n+1) + \frac{k+n+1}{2k+2n+1}$, and (2) we obtain the decomposition of ordinary $V^k(\mathfrak{osp}_{1|2n})$-modules into $V^k(\mathfrak{sp}_{2n}) \otimes \mathcal W^\ell(\mathfrak{sp}_{2n})$-modules. Next, if $k$ is an admissible level and $\ell$ is a non-degenerate admissible level for $\mathfrak{sp}_{2n}$, we show that the simple algebra $L_k(\mathfrak{osp}_{1|2n})$ is an extension of the simple subalgebra $L_k(\mathfrak{sp}_{2n}) \otimes {\mathcal W}_{\ell}(\mathfrak{sp}_{2n})$. Using the theory of vertex superalgebra extensions, we prove that the category of ordinary $L_k(\mathfrak{osp}_{1|2n})$-modules is a semisimple, rigid vertex tensor supercategory with only finitely many inequivalent simple objects. It is equivalent to a certain subcategory of  $\mathcal W_\ell(\mathfrak{sp}_{2n})$-modules. A similar result also holds for the category of Ramond twisted modules. Due to a recent theorem of Robert McRae, we get as a corollary that categories of ordinary $L_k(\mathfrak{sp}_{2n})$-modules are rigid.
\end{abstract}


\section{Introduction}

The simple Lie superalgebra $\mathfrak{osp}_{1|2n}$ behaves in many respects like a simple Lie algebra. For example the category of finite-dimensional representations of a simple  Lie superalgebra $\g$ is semisimple if and only if $\g$ is of type $\mathfrak{osp}_{1|2n}$ or a Lie algebra \cite{DH}. It is also the only simple, basic Lie superalgebra that is not a Lie algebra and that has no isotropic roots and that admits a positive definite invariant, consistent, supersymmetric bilinear form \cite{Scheunert, Kac77}. Maybe most remarkable is a very close connection to the representation theory of $\mathfrak{so}_{2n+1}$. Both algebras have the same set of simple roots and there is a one-to-one correspondence between simple finite dimensional representations, such that characters and tensor products agree under this correspondence \cite{RS}.

Vertex algebras arose from the physics literature as the chiral algebras of two-dimensional conformal field theory \cite{CFT}. In particular affine vertex algebras are those that correspond to the Wess-Zumino-Witten theories on Lie groups of physics \cite{Witten83}. Let $\g$ be a finite-dimensional Lie superalgebra over $\mathbb C$ and $B: \g \times \g \rightarrow \mathbb C$ an invariant, consistent and supersymmetric bilinear form. Let
\[
\widehat \g = \mathbb C[t, t^{-1}] \otimes_{\mathbb C} \g \oplus \mathbb C K
\]
with $K$ central and 
\[
[x \otimes t^n, y \otimes t^m] = [x, y] \otimes t^{n+m} = \delta_{n+m, 0} B(x, y)K
\]
for $x, y \in \g$ and $n, m \in \mathbb Z$. Let $\widehat \g_{\geq 0} = \mathbb C[t] \otimes_{\mathbb C} \g \oplus \mathbb C K$ and let $\mathbb C_k$ be the one-dimensional $\widehat \g_{\geq 0}$ on which $K$ acts by multiplication with the level $k \in \mathbb C$, while $ \mathbb C[t] \otimes_{\mathbb C} \g$ acts as zero. One sets $V^k(\g, B) :=  \text{Ind}^{\widehat g}_{\widehat g_{\geq 0}} \mathbb C_k$. $V^k(\g, B)$ can be given the structure of a vertex algebra, the affine vertex (super)algebra of $\g$ at level $k$. $V^k(\g, B)$ might not be simple and one denotes by $L_k(\g, B)$ its simple quotient. If $\g$ admits a unique non-degenerate  positive, invariant, consistent and supersymmetric bilinear form, then one normalizes it such that long roots have norm two and just writes $V^k(\g)$ and $L_k(\g)$. Representations of $L_k(\g)$ are in particular smooth $\widehat \g$-modules at level $k$, but being a vertex algebra module is sometimes very restrictive and for special levels $k$ only few $\widehat \g$-modules  are also $L_k(\g)$-modules. The most famous instance is the case where $\g$ is a simple Lie algebra and $k$ is a positive integer. In this case every $L_k(\g)$-module is ordinary and its representation category is a modular tensor category, that is a finite, non-degenerate and semisimple ribbon category \cite{H-mod, H-ver}. Via the Sugawara construction \cite{Sugawara:1967rw} an affine vertex algebra at non-critical level contains the Virasoro vertex algebra as subalgebra and in particular every $L_k(\g)$-module is graded by conformal weight (generalized eigenspaces corresponding to the Cartan subalgebra of the Virasoro Lie algebra). A module is called ordinary if its conformal weight spaces are finite-dimensional and the conformal weight is lower-bounded. 
In particular every conformal weight space of an ordinary module of $L_k(\g)$ is integrable for the horizontal subalgebra $\g$. As an example, let $E_\lambda$ be an integrable module of highest-weight $\lambda$ of $\g$. Extend it to $\g_{\geq 0}$ by letting $K$ act by mutiplication with $k$ and $t \g[t]$ by $0$. Then $V^k(\lambda) :=  \text{Ind}^{\widehat g}_{\widehat g_{\geq 0}} E _\lambda$ is an ordinary module, the Weyl-module of highest-weight $\lambda$ at level $k$. We will denote its simple quotient by $L_k(\lambda)$.  
 In this work we are interested in the category of ordinary modules, $\cO_k(\g)$, for $\g = \mathfrak{osp}_{1|2n}$.

Given the similarity between $\mathfrak{osp}_{1|2n}$ and simple Lie algebras it is reasonable to expect that the representation theory of affine vertex algebras of $\mathfrak{osp}_{1|2n}$ is similar to that of affine vertex algebras of simple Lie algebras. Indeed the simple affine vertex algebra of $\mathfrak{osp}_{1|2n}$ at positive integral level $n$ is rational and lisse \cite[Theorem 7.1]{CL21}, while no other simple affine vertex superalgebra at non-zero level can be lisse \cite{GK}. Rational and lisse are the finiteness properties that guarantee that the category of ordinary modules is a modular tensor category \cite{H-mod}. Another special class of levels are the admissible ones \cite{KW89} and  
there are quite a few interesting results on ordinary modules of affine vertex algebras at admissible levels.  Let $\g$ be a simple, finite dimensional Lie algebra, $k$ an admissible level and $\cO_k(\g)$ the category of ordinary modules of $L_k(\g)$. Then
\begin{enumerate}
\item $\cO_k(\g)$ is semisimple with a finite number of inequivalent simple objects \cite{Ara-adm}.
\item  $\cO_k(\g)$  is a braided tensor category \cite{CHY}.
\item If $\g$ is simply-laced then $\cO_k(\g)$ is a fusion category \cite{C-fus}.  
\item $\cO_k(\g)$ is a fusion category if $\g$ is of  type $B$ and if the denominator of the admissible level is $2$  \cite{CKoL}.
\item If $\g$ is simply-laced, then a simple current modification of $\cO_k(\g)$ is braided tensor equivalent to a subcategory of the principal $\W$-algebra of $\g$ at level $1-h^\vee + \frac{1}{k+h^\vee}$ with $h^\vee$ the dual Coxeter number of $\g$ \cite{C-fus}.  
\end{enumerate}
$\mathcal W$-algebras are vertex algebras obtained from affine ones via quantum Hamiltonian reduction \cite{FF90, KRW}.
We prove analogues of all these results for $\g = \mathfrak{osp}_{1|2n}$ as well as that $\cO_k(\mathfrak{sp}_{2n})$ (for many admisible levels) is a fusion category. The proof goes in two steps. First, we realize  $V^k(\mathfrak{osp}_{1|2n})$ from the equivariant $\mathcal W$-algebra of $\mathfrak{sp}_{2n}$ and second, we pass to interesting rational levels and use the theory of vertex superalgebra extensions. 

\subsection{ $V^k(\mathfrak{osp}_{1|2n})$ from the equivariant $\mathcal W$-algebra of $\mathfrak{sp}_{2n}$}

The starting point is the algebra of chiral differential operators (CDOs) $\CDO_{G, \kappa}$ at level $\kappa$, for $G= Sp(2n)$ \cite{MSV,GMS}.
While one thinks about affine vertex algebras as chiralizations of finite dimensional Lie algebras, the CDO chiralizes the space of functions of an algebraic group. In particular
 this is a vertex algebra with an action of $V^{k}(\g) \otimes V^{\ell}(\g)$ with $\g$ the Lie algebra of $G$ and $\kappa = k+ h^\vee$, $-\kappa = \ell + h^\vee$. For generic level $\kappa$, it satisfies
\[
\CDO_{G, \kappa} \cong \bigoplus_{\lambda \in P^+} V^k(\lambda) \otimes V^\ell(\lambda^*).
\]
The $V^k(\lambda)$ are Weyl modules at level $k$ of highest weight $\lambda$ and $P^+$ is the set of dominant weights of $\g$.
The equivariant (principal) $\W$-algebra is obtained from $\CDO_{G, \kappa}$ by quantum Hamiltonian reduction with respect to the principal nilpotent element in the first factor \cite{Arakawa:2018egx}. For $\kappa$ generic. 
\[
\W_{G, \kappa} \cong \bigoplus_{\lambda \in P^+} T^{\kappa}_{\lambda, 0} \otimes V^\ell(\lambda^*), \qquad T^{\kappa}_{\lambda, 0} = H^0_{\text{DS}}(V^k(\lambda))
\]
Let $F^{2n+1}$  be the vertex superalgebra of $2n+1$ free fermions. We consider $\W_{G,  \kappa}  \otimes F^{2n+1}$. 
Recall Feigin-Frenkel duality of type $B, C$, that is $\W^k(\mathfrak{sp}_{2n}) \cong W^{\check{k}}(\mathfrak{so}_{2n+1})$ with $2\kappa \check{\kappa} =1$ and $\check{\kappa} = \check{k} + 2n-1$.  
Quantum Hamiltonian reduction can be twisted by spectral flow; the corresponding functor is denoted by  $H^0_{\text{DS}, \mu}$ for $\mu$ a coweight of $\g$. A special case of Arakawa-Frenkel duality \cite[Thm. 2.2]{AF19} is $T^{\kappa}_{\lambda, 0}  \cong H^0_{\text{DS}, \lambda}(V^{\check{k}}(\mathfrak{so}_{2n+1}))$. The Urod property is that the twisted reduction functor commutes with tensoring with integrable representations  \cite[Thm. 7.2]{ACF}. Applying this to $V^{\check{k}}(\mathfrak{so}_{2n+1}) \otimes F^{2n+1}$ gives us an explicit decomposition of $T^{\kappa}_{\lambda, 0}  \otimes F^{2n+1}$ into modules of $W^{\check{k}+1}(\mathfrak{so}_{2n+1}) \otimes \W^t(\mathfrak{osp}_{1|2n})$. Here the principal $\W$-algebra of $\mathfrak{osp}_{1|2n}$ at a certain level $t$ related to $k$ appears due to the coset realization $ \W^t(\mathfrak{osp}_{1|2n}) \cong \text{Com}(V^{\check{k}+1}(\mathfrak{so}_{2n+1}), V^{\check{k}}(\mathfrak{so}_{2n+1}) \otimes F^{2n+1})$ \cite{CL21}.
Inserting this into 
$\W_{G, \kappa} \otimes F^{2n+1}$ gives us an explicit decomposition of the form (for $\kappa$ generic)
\begin{equation}
\W_{Sp(2n), \kappa} \otimes F^{2n+1}  \cong \bigoplus_{\lambda \in R} C^\ell_{\lambda} \otimes N^t_{\lambda, 0}
\end{equation}
with $C^\ell = \text{Com}(\W^t(\mathfrak{osp}_{1|2n}), \W_{Sp(2n), \kappa} \otimes F^{2n+1} )$ with $ C^\ell_{\lambda}$ and $N^t_{\lambda, 0}$ certain $C^\ell$ and $\W^t(\mathfrak{osp}_{1|2n})$-modules. Here $R$ denotes the set of non-spin representations of $\mathfrak{so}_{2n+1}$ which is naturally identified with the set of dominant weights of $\mathfrak{osp}_{1|2n}$. Under this identification we have 
\begin{theorem}[Theorem \ref{thm1}]
For generic $\ell$,  $C^\ell \cong V^\ell(\mathfrak{osp}_{1|2n})$ and $C^\ell_\lambda \cong V^\ell(\lambda)$. 
\end{theorem}
The proof of this theorem can be reduced to a character computation that we do in the appendices as it is quite technical. By construction we get that the coset 
$\text{Com}(V^\ell(\mathfrak{sp}_{2n}), V^\ell(\mathfrak{osp}_{1|2n}))$ is $\mathcal W^s(\mathfrak{sp}_{2n})$  with  $\frac{1}{\ell+n+1} + \frac{s}{n+1}=2$ (this was first proven in \cite{CL21}) as well as the decomposition of $V^\ell(\lambda)$ into $V^\ell(\mathfrak{sp}_{2n}) \otimes \mathcal W^s(\mathfrak{sp}_{2n})$-modules, 
see Theorem \ref{CDOdecomposition}. The above construction can be generalized substantially and in particular leads to a new and quick proof of the coset realization of principal $\mathcal W$-algebras of type ADE \cite{CN22}.

We turn now to the category of ordinary modules at admissible levels. 

\subsection{The category of ordinary modules of $\mathfrak{osp}_{1|2n}$ at admissible levels}

While the above results hold for generic level, we are now interested in special rational levels, namely let $k = -h^\vee + \frac{u}{v}$ be admissible and $\ell = -h^\vee + \frac{u}{2u-v}$ be non-degenerate admissible for $\mathfrak{sp}_{2n}$. Gorelik and Serganova extended the notion of admissible level to the super case, see in particular section 5 of \cite{GS}. We determine them in Proposition  \ref{prop:osp-adm} and note that the set of levels that we are interested is slightly larger than just the admissible ones of \cite{GS}, see Remark \ref{rem:admissible}.

 In particular we want the affine subalgebra of type $\mathfrak{sp}_{2n}$ to be at admissible level and the $\W$-algebra to be at non-degenerate admissible level. We then can prove that the simple subalgebras act on the simple affine VOA $L_\ell(\mathfrak{osp}_{1|2n})$ and this gives us the full power of the theory of vertex superalgebra extensions. 
Using this theory, we get our main results, Theorem \ref{thm2}] and Corollaries  \ref{cor1} and  \ref{cor2}. 

That is for $k = -h^\vee + \frac{u}{v}$ admissible and $\ell = -h^\vee + \frac{u}{2u-v}$ non-degenerate admissible for $\mathfrak{sp}_{2n}$:
\begin{enumerate}
\item A complete list of inequivalent simple objects in $\cO_k(\mathfrak{osp}_{1|2n})$. 
\item A complete list of inequivalent simple objects in the category of ordinary Ramond twisted modules of $L_k(\mathfrak{osp}_{1|2n})$.
\item The category $\cO_k(\mathfrak{osp}_{1|2n})$ is a semisimple, rigid vertex tensor supercategory, i.e. it is a fusion supercategory.
\item The category of ordinary Ramond twisted modules of $L_k(\mathfrak{osp}_{1|2n})$ is semisimple.
\item $\cO_k(\mathfrak{sp}_{2n})$ is rigid.
\item There is a super-braid-reversed equivalence $\cO_k(\mathfrak{sp}_{2n}) \cong \cC(\W_\ell(\mathfrak{sp}_{2n}))_L$. 
\item There is a super-braided equivalence $\cC(\W_\ell(\mathfrak{sp}_{2n}))_R^Q \cong  \cO_k(\mathfrak{osp}_{1|2n})$ . 
\end{enumerate}
Here $ \cC(\W_\ell(\mathfrak{sp}_{2n}))_L$ and $\cC(\W_\ell(\mathfrak{sp}_{2n}))_R^Q$ are certain full tensor subcategories of the $\W$-algebra and are introduced in Section \ref{sec:adm}.

At admissible level, the semisimplicity of ordinary modules and a complete list of simple objects has already been obtained by Gorelik and Serganova \cite{GS}. 

\subsection{Outline} 

The paper is organized as follows. In Section 2.1, we collect the necessary results on vertex tensor supercategories. In Section 2.2 and 2.3, we introduce admissible levels and several notations for subsets of admissible weights of Lie algebras and show the isomorphisms \eqref{eq:v=1} and \eqref{eq:v=2}. In Section 2.3, we introduce the modules $T^\kappa_{\lambda, \mu}$ and $\WL_k(\lambda, \mu)$ of principal $\mathcal{W}$-algebras, recall the fusion rules and some properties of the subcategories $\mathcal C_L(p, q)$ and $\mathcal C_R(p, q)$ of the category of $\mathcal{W}_k(\mathfrak{g})$-modules  from \cite{C-fus}; see Remark \ref{rmk:prop}. In Section 3, we recall the equivariant $\mathcal{W}$-algebras, prove Theorem \ref{CDOdecomposition}  by using results in \cite{AF19, ACF}, and prove Theorem \ref{thm1} by using Theorem \ref{CDOdecomposition} and Theorem \ref{thm:character}. In Section 4, we prove Theorem \ref{thm2} by using results in Section 2.1, \eqref{eq:v=1}, \eqref{eq:v=2} and Remark \ref{rmk:prop}, and derive Corollary \ref{cor1} and Corollary \ref{cor2}. In the Appendix, we compute the characters of modules of principal $\mathcal{W}$-algebras of type $C$ and show Theorem \ref{thm:character}.

\noindent{\bf Acknowledgements} We thank Shigenori Nakatsuka for very helpful discussions. The work of T.C. is supported by NSERC Grant Number RES0048511. N.G. is supported by World Premier International Research Center Initiative (WPI), MEXT, Japan and JSPS KAKENHI Grant Number JP21K20317. A.L. is supported by Simons Foundation Grant 635650 and NSF Grant DMS-2001484.

\section{Preliminaries}

\subsection{Vertex tensor category theory}

This section explains useful results on the theory of vertex superalgebra extensions \cite{CKM1}. A good summary of the most important theorems is Section 2.4 \cite{CY-adm}.

Let $V$ be a vertex operator algebra and $\cC_V$ be a semisimple category of $V$ modules that is a vertex tensor category in the sense of \cite{HLZ1, HLZ2, HLZ3, HLZ4, HLZ5, HLZ6, HLZ7, HLZ8}. Denote by $\text{Irr}(\cC_V)$ the set of inequivalent simple objects in $\cC_V$. Let $W$ be a strongly rational vertex operator algebra; in particular, its category $\cC_W$ of $W$-modules is a modular tensor category \cite{H-mod, H-ver}. 
The dual of an object $M$ is denoted by $M^*$. 
Assume that $A$ is a simple vertex operator superalgebra extending $V\otimes W$ and 
\begin{equation}\label{eq:decA}
A \cong  \bigoplus_{X \in \text{Irr}(\cC_V)} X \otimes \tau(X)^*
\end{equation}
as a $V \otimes W$-module, where the $\tau(X)$ are inequivalent simple $W$-modules. Denote by $\cD_W$ the subcategory of $W$-modules whose objects are direct sums of simple modules appearing in the decomposition of $A$. 
We assume that $V\otimes W$ is a subalgebra of the even subalgebra $A_0$ of $A$. 
This set-up is assumed to hold throughout this section. 
\begin{theorem}\label{thm:rigid} \textup{(Special case of Theorem 4.9 of \cite{mcrae2021rigid})}

\noindent $\cC_V$ is rigid and super-braid-reversed equivalent to the subcategory $\cD_W$ of $\cC_W$ whose simple objects are the $\tau(X)$ for $X \in \text{Irr}(\cC_V)$.
\end{theorem}

Set $\cC :=\cC_V \boxtimes \cC_W$. Note that the Deligne product of vertex tensor categories is a vertex tensor category provided that at least one of the two categories is semisimple \cite[Thm. 5.5]{CKM2}. We have just seen that $\cC$ is rigid. 
There is a $\cC$-module category associated to $A$, denoted by $\cC_A$. A certain subcategory of local modules, denoted by $\cCloc$, is in fact a braided tensor category itself. One of the main results of \cite{CKM1} is that this category is equivalent as a braided tensor category to the category of modules of the vertex operator superalgebra $A$ that lie in $\cC$. Moreover there is a functor $\cF$ from $\cC$ to $\cC_A$ and $\cF(X)$ for $X$ an object in $\cC$ is local if and only if the monodromy $M_{A, X} = c_{X, A} \circ c_{A, X}$ of $X$ with $A$ is trivial.  Here $c_{\bullet, \bullet}$ denotes the braiding in $\cC$. 
The induction functor has the important property that it preserves duality \cite[Prop. 2.77]{CKM1}, in particular a rigid object induces to a rigid object. 
Similarly we have the category of modules and local modules for $A_0$ and corresponding induction functor $\cF_{A_0}$. There is then also an induction functor, call it $\cF_0$ from local $A_0$-modules to $A$-modules. For a simple object $M$ in $\cC_{A_0}^{\text{loc}}$ the monodromy with $A_1$ is either one or minus one, since monodromy respects tensor products \cite[Thm. 2.11]{CKL}.
The Ramond twisted 
$A$-modules are objects in $\cC_A$ that are local as $\cC_{A_0}$-modules and for which the monodromy with $A_1$ is minus one. 
The monodromy respects tensor products, in particular if $X, Y$ have each monodromy minus one with $A_1$, then $X \boxtimes Y$ has trivial monodromy. Since induction is monoidal and every object is a subquotient of an induced object this means that the tensor product of two Ramond twisted modules is always local. The same type of argument yields that the tensor product of a local module with a Ramond twisted module is always Ramond twisted. 
Proposition 2.77 of \cite{CKM1} says that $\cF(X)^* \cong \cF(X^*)$. Moreover
Lemma 2.78 of \cite{CKM1} states that if $\cF(X)$ is local then so is $\cF(X^*)$ and the same proof with an additional minus sign gives that if $\cF(X)$ is Ramond twisted, then so is $\cF(X^*)$. 

The category of Ramond twisted modules is denoted by $\cCtw$ and we will sometimes omit the word Ramond. By Yoneda's Lemma one has $\cF = \cF_{A_0} \circ \cF_0$, see Lemma 9.1 of \cite{ACK21} for a proof. In particular $\cF(X)$ is in $\cCtw$ if $M_{A_0, X} = \text{Id}_{A_0 \boxtimes X}$ and $M_{A_1, X} = -\text{Id}_{A_1 \boxtimes X}$. We summarize properties as just explained
\begin{remark} \label{rmk:proptw} ${}$
\begin{enumerate}
\item $\cF(X)$ is rigid if $X$ is rigid. The dual is $\cF(X)^* \cong \cF(X^*)$ and if $\cF(X)$ is local (resp. Ramond twisted) then so is $\cF(X)^*$.
\item $\cCtw$ is a $\cCloc$-module category.
\item  The tensor product of two Ramond twisted modules is always local.
\item $\cF(X)$ for $X$ in $\cC_A$ is Ramond twisted if $M_{A_0, X} = \text{Id}_{A_0 \boxtimes X}$ and $M_{A_1, X} = - \text{Id}_{A_1 \boxtimes X}$.
\end{enumerate}
\end{remark}

\begin{theorem}\label{thm:semisimple} \textup{(Prop. 2.1 of \cite{CL21} or equivalently Thm. 5.2 and Cor. 5.3 of \cite{mcrae2021semisimplicity})}
\noindent The categories of local and Ramond twisted modules of $A$ are semisimple.
\end{theorem}

We continue with the same set-up as before. 
Let us set $N_X :=  X \otimes \tau(X)^*$. The multiplication rule $M_{X, Y}^{Z}$ is defined to be one if $N_Z$ appears in the operator product of $N_X$ with $N_Y$, and it is zero otherwise. This is rephrased in categorical terms  in Definition 2.5 of \cite{CKM2}. For an algebra in a rigid vertex tensor category of the form \eqref{eq:decA}, Theorem 3.5 of \cite{CKM2} says that $M_{X, Y}^{Z}=1$ if and only if $\tau(Z)^*$ is a summand of $\tau(X)^* \boxtimes \tau(Y)^*$. This statement is only proven if $A$ is an algebra in \cite{CKM2}, however the argument (which is Section 3.3 of \cite{CKM2}) is exactly the same in the superalgebra case. In particular, $N_X$ together with $V \otimes W$ generates $A$ under operator products if and only if $\tau(X)^*$ generates $\cD_W$ as a fusion ring. 
	\begin{lemma}\label{lem:gen}
\textup{(Consequence of Main Theorem 3.5 of \cite{CKM2})}
\noindent $A$ is generated under operator products by $V \otimes W$ together with any field in $X \otimes \tau(X)^*$ if and only if $\cD_W$ is generated as a fusion ring by the object $\tau(X)^*$.
\end{lemma}

Finally a very useful theorem is Frobenius reciprocity. 
\begin{theorem}\label{thm:Frec}\textup{(Lemma 2.61 of \cite{CKM1})}

\noindent
Let $\cG: \cC_A \rightarrow \cC$ be the forgetful functor, then for any object $X$ in $\cC$ and $Y$ in $\cC_A$, the two spaces
\[
\text{Hom}_\cC\left(X, \cG(Y)\right) \cong \text{Hom}_{\cC_A}\left(\cF(X), Y\right)
\]
are naturally isomorphic.
\end{theorem}

\subsection{Admissible levels for $\mathfrak{g} = \mathfrak{osp}_{1|2n}$}

Let $\mathfrak g$ be a Lie superalgebra equipped with a non-degenerate, supersymmetric, invariant, even bilinear form $B$ on $\mathfrak g$. Then, to this data and each complex number $k \in \mathbb C$, one can associate the affine vertex superalgebra $V^k(\mathfrak g, B)$. We are interested in the cases where $\mathfrak g$ is either a simple Lie algebra or $\mathfrak g = \mathfrak{osp}_{1|2n}$. In these instances, we just write $V^k(\mathfrak g)$ when  $B$ is normalized such that long roots have norm two. The simple quotient of $V^k(\mathfrak g)$ is denoted by $L_k(\mathfrak g)$.

The notion of admissible level is due to Kac and Wakimoto, who computed them explicitly in Proposition 1.1 of \cite{KW08} for simple Lie algebras. 
\begin{proposition} {\textup{ (\cite[Prop.1.1]{KW08})}}
Let $\g$ be a simple Lie algebra over $\C$ of rank $n$ and let $k \in\C$. Let $( \cdot | \cdot )$ be the Killing form on $\g$, normalized such that long roots have norm two. 
 Let $h, h^\vee$ be the Coxeter and dual Coxeter number of $\g$ and $r^\vee$ the lacity, that is $r^\vee =1$ in type $A, D, E$, $r^\vee =2$ in type $B, C, F_4$ and $r^\vee = 3$ for $G_2$. A level $k$ is called admissible if 
\begin{equation}\nonumber
k + h^\vee = \frac{p}{q} \in  \Q_{>0}, \quad (p, q) =1, \quad p, q >0 \ \ \text{and} \ \ \begin{cases} p \geq h^\vee & \ \ (q, r^\vee) =1, \\  p \geq h & \ \ (q, r^\vee) = r^\vee . \end{cases}
\end{equation}
The first case is called principal admissible and the second one coprincipal. A principal admissible level is non-degenerate if $q \geq h$ and a coprincipal one is non-degenerate if $q \geq r^\vee h^\vee$. 
\end{proposition}

The definition is
\begin{definition} \textup{\cite{KW89} (see also Definition 1.1 of \cite{KW08})} A weight $\lambda \in \widehat{h}^*$ is called admissible if 
$(\lambda + \widehat\rho, \alpha^\vee) > 0$ for all $\alpha^\vee \in \widehat{\Delta}^\vee_+$ and the $\mathbb Q$-span of $\widehat{\Delta}^\vee_\lambda$ contains $\widehat{\Delta}^\vee$.
\end{definition}
The notation will be reviewed in a moment. 
A level $k$ is then called admissible if the weight $k \Lambda_0$ is admissible and $\Lambda_0$ is the fundamental weight corresponding to the $0$-th root $\alpha_0$ of the affine Lie algebra. 
Gorelik and Serganova extended the definition to affine Lie superalgebras \cite{GS}, see in particular section 5 for the case of $\g = \mathfrak{osp}_{1|2n}$. We compute these explicitly in the appendix.
\begin{proposition} \textup{(Proposition \ref{prop:osp-adm})}
For $k \in \C$, $\lambda = k \Lambda_0$ is admissible if and only if there exist coprime integers $a \in \Z$ and $b \in \Z_{\geq1}$ such that $\displaystyle k= \frac{a}{b}$ and
\begin{enumerate}
\item $\displaystyle k + h^\vee \geq \frac{n+\frac{1}{2}}{b}$\quad if $(b, 2) = 1$
\item $\displaystyle k + h^\vee \geq \frac{2n-1}{b}$\quad if $(b, 2) = 2$.
\end{enumerate}
\end{proposition}

\begin{remark}\label{rem:admissible}
Let $\displaystyle k = \frac{a}{b}$ be an admissible level for $\mathfrak{osp}_{1|2n}$. Define $u, v$ via $\displaystyle k = -(n+1) + \frac{u}{v}$ and $(u, v) = 1$, and set $\displaystyle \ell = -(n+1) + \frac{u}{2u-v}$, in particular $v=b$, $u = (k+n+1)b$, and $\displaystyle 2u-v = 2\left(k+n+\frac{1}{2}\right)b$.

We are interested in the cases where $k$ is admissible for $\mathfrak{sp}_{2n}$ and $\ell$ is non-degenerate admissible. 

Consider the case $b$ is odd. Then $\displaystyle k+ n + \frac{1}{2} \geq \frac{n+\frac{1}{2}}{b}$ and so $\displaystyle u = (k+n+1)b \geq n+\frac{1}{2}+\frac{b}{2} \geq n+1$ and $\displaystyle 2u-v =2\left(k+n+\frac{1}{2}\right)b \geq 2n+1 >  2n$.
 So $k$ is principal admissible and $\ell$ is non-degenerate principal admissible for
$\mathfrak{sp}_{2n}$. 

Similarly for $b$ even. Then $\displaystyle k+ n + \frac{1}{2} \geq \frac{2n-1}{b}$  and so $\displaystyle u \geq 2n$ and $2u-v \geq 4n-2$ ( $\geq 2(n+1)$ if $n \geq 2$). So $k$ is coprincipal admissible and $\ell$ is non-degenerate coprincipal admissible for
$\mathfrak{sp}_{2n}$. 

If $b$ is odd, then $k$ being admissible  for $\mathfrak{osp}_{1|2n}$ is in fact equivalent to $k$  principal admissible and $\ell$  non-degenerate principal admissible for
$\mathfrak{sp}_{2n}$. However if $b$ is even and if $n>1$, then there are some levels $k$, such that $k$ is not admissible for $\mathfrak{osp}_{1|2n}$ but still is
coprincipal admissible and $\ell$ is non-degenerate coprincipal admissible for
$\mathfrak{sp}_{2n}$. 
\end{remark}

\subsection{Admissible weights for Lie algebras}\label{sec:LA}

Let $\Pi = \{ \alpha_1, \dots, \alpha_n\}$ be the set of positive simple roots, normalized such that the longest root has norm $2$. Let $Q$ be the root lattice. The fundamental coweights $\omega_1^\vee, \dots, \omega_n^\vee$ are the duals of the positive simple roots and they span the coweight lattice $P^\vee$. The coroots are defined by 
$\alpha^\vee = \frac{2}{(\alpha|\alpha)} \alpha$
and the coroot lattice is denoted by $Q^\vee$. Its dual is the weight lattice $P$, spanned by the fundamental weights $\omega_1, \dots, \omega_n$. 
Let $\theta, \theta_s^\vee$ be the longest root and longest short coroot. 
Let $k + h^\vee = \frac{p}{q}$ be admissible, define the sets of admissible weights
\begin{equation}\nonumber
\begin{split}
P(p, q) &:= \left\{ \lambda \in P \ | \ (\lambda | \alpha_i^\vee) \in \Z_{\geq 0}, \ i =1, \dots, n,  \ (\lambda | \theta) \leq  p-h^\vee \right\}, \quad \text{if} \ (q, r^\vee) =1, \\
P(p, q) &:= \left\{ \lambda \in P \ | \ (\lambda | \alpha_i^\vee) \in \Z_{\geq 0}, \ i =1, \dots, n,  \ (\lambda | \theta_s^\vee) \leq  p-h \right\}, \quad \text{if} \ (q, r^\vee) =r^\vee. 
\end{split}
\end{equation}
Let ${}^L\g$ be the Langlands dual Lie algebra of $\g$. This means the roots of ${}^L\g$ are the coroots of $\g$ and vice versa. 
We denote roots, coroots, weights, coweights and their lattices and all other quantities associated to ${}^L\g$ by an additional symbol ${}^L$. 
Since long roots are normalized to have norm $2$ and the short roots have norm $2/r^\vee$ one has the relations
\[
{}^LQ^\vee = \sqrt{r^\vee}Q \qquad \text{and} \qquad {}^LP^\vee = \sqrt{r^\vee}P. 
\]
The dual level ${}^Lk$ to $k$ is defined by 
\[
r^\vee ({}^Lk+{}^Lh^\vee)(k+h^\vee) = 1.
\]
In particular, if $k+h^\vee = \frac{p}{q}$ is non-degenerate principal admissible, then ${}^Lk+{}^Lh^\vee = \frac{q}{r^\vee p}$ is coprincipal for ${}^L\g$. 
Define the sets 
\begin{equation}\nonumber
\begin{split}
{}^LP^\vee(p, q) &:= \left\{ \lambda \in {}^LP^\vee \ | \ (\lambda | \alpha_i) \in \Z_{\geq 0}, \ i =1, \dots, n,  \ (\lambda | {}^L\theta^\vee_s) \leq r^\vee( p-h^\vee) \right\}, \quad \text{if} \ (q, r^\vee) =1, \\
{}^LP^\vee(p, q) &:= \left\{ \lambda \in {}^LP^\vee \ | \ (\lambda | \alpha_i) \in \Z_{\geq 0}, \ i =1, \dots, n,  \ (\lambda | {}^L\theta) \leq  p-h \right\}, \quad \text{if} \ (q, r^\vee) =r^\vee. 
\end{split}
\end{equation}
Since ${}^L\theta_s^\vee = \sqrt{r^\vee}\theta$ rescaling by $\sqrt{r^\vee}$ gives the isomorphism
\[
P(p, q)  \cong {}^LP^\vee(p, q).
\]
In particular, if $k$ is principal admissible, then (note that $h = {}^Lh$ and $h \geq h^\vee$ for any simple Lie algebra)
\begin{equation}\label{eq:embweights}
\begin{split}
{}^LP(q, r^\vee p) &\cong P^\vee(q, r^\vee p) \\
&=  \left\{ \lambda \in P^\vee \ | \ (\lambda | \alpha_i) \in \Z_{\geq 0}, \ i =1, \dots, n,  \ (\lambda | \theta) \leq  q-h \right\} \\
&\subset \left\{ \lambda \in P \ | \ (\lambda | \alpha_i^\vee) \in \Z_{\geq 0}, \ i =1, \dots, n,  \ (\lambda | \theta) \leq  q-h \right\} \\
&\subset \left\{ \lambda \in P \ | \ (\lambda | \alpha_i^\vee) \in \Z_{\geq 0}, \ i =1, \dots, n,  \ (\lambda | \theta) \leq  q-h^\vee \right\} 
= P(q, 1). 
\end{split}
\end{equation}

Finally, we define the subset of admissible weights that lie in the root lattice $P_Q(p, q) := P(p, q) \cap Q$. 
We specialize to type $B_n, C_n$. Let $\epsilon_1, \dots, \epsilon_n$ be an orthonormal basis of $\Z^n$ and we view roots as embedded in $\Z^n$ or a rescaling of it. Some data for $B_n$ is 
\begin{enumerate}
\item simple roots $\alpha_1 = \epsilon_1 -\epsilon_2, \dots, \epsilon_{n-1}-\epsilon_n, \epsilon_n$,
\item simple coroots $\alpha_1^\vee = \alpha_1, \dots, \alpha_{n-1}^\vee=\alpha_{n-1}, \alpha_n^\vee =2\alpha_n$,
\item longest root $\theta=\epsilon_1 +\epsilon_2$,
\item longest short coroot $\theta_s^\vee = 2\epsilon_1$,
\item Coxeter number $h=2n$, dual Coxeter number $h^\vee = 2n-1$. 
\end{enumerate}
The corresponding data for $C_n$ is 
\begin{enumerate}
\item simple roots $\alpha_1 = \frac{\epsilon_1 -\epsilon_2}{\sqrt{2}}, \dots, \frac{\epsilon_{n-1}-\epsilon_n}{\sqrt{2}}, \sqrt{2}\epsilon_n$,
\item simple coroots $\alpha_1^\vee = 2\alpha_1, \dots, \alpha_{n-1}^\vee=2\alpha_{n-1}, \alpha_n^\vee =\alpha_n$,
\item longest root $\theta= \sqrt{2}\epsilon_1$,
\item longest short coroot $\theta_s^\vee = \sqrt{2}(\epsilon_1+\epsilon_2)$,
\item Coxeter number $h=2n$, dual Coxeter numbert $h^\vee = n+1$. 
\end{enumerate}
We want to compare the sets of admissible weights of $B_n$ and $C_n$, so we add a superscript $B, C$ to indicate the type. Let $p\geq n+1$, then
\begin{equation}\nonumber
\begin{split}
P^C(p, 1) &= \left\{ \lambda \in \frac{1}{\sqrt{2}} \mathbb Z \ | \ (\lambda | \sqrt{2}(\epsilon_i-\epsilon_{i+1})) \in \Z_{\geq 0}, \ i =1, \dots, n-1, \  (\lambda | \sqrt{2}\epsilon_n) \in \Z_{\geq 0},\frac{}{} \ (\lambda | \sqrt{2}\epsilon_1) \leq p-n-1 \right\}, \\
P^B(2p-1, 2p) &= \left\{ \mu \in  \mathbb Z \ | \ ( \mu | \epsilon_i-\epsilon_{i+1}) \in \Z_{\geq 0}, \ i =1, \dots, n-1, \  ( \mu | 2\epsilon_n) \in \Z_{\geq 0}, \frac{}{} \ ( \mu 2 | \epsilon_1) \leq 2p-2n-1 \right\}, \\
P^B_Q(2p-1, 2p) &= \left\{ \mu \in  \mathbb Z \ | \ (\mu | \epsilon_i-\epsilon_{i+1}) \in \Z_{\geq 0}, \ i =1, \dots, n-1, \  (\mu | \epsilon_n) \in \Z_{\geq 0}, \frac{}{} \ (\mu | \epsilon_1) \leq p-n-\frac{1}{2} \right\} .\\
 \end{split}
\end{equation}
We observe that the map $\lambda \mapsto \sqrt{2}\lambda$ provides the isomorphism
\begin{equation}\label{eq:v=1}
P^C(p, 1) \cong P^B_Q(2p-1, 2p).
\end{equation}
Next, let $p \geq 2n$ be odd. Then 
\begin{equation}\nonumber
\begin{split}
P^C(p, 2) &= \left\{ \lambda \in \frac{1}{\sqrt{2}} \mathbb Z \, | \, \left(\lambda  | \sqrt{2}(\epsilon_i-\epsilon_{i+1})\right) \in \Z_{\geq 0}, \ i =1, \dots, n-1, \,  (\lambda |  \sqrt{2}\epsilon_n) \in \Z_{\geq 0}, \, \left(\lambda | \sqrt{2}(\epsilon_1+\epsilon_2)\right) \leq p-2n \right\}, \\
P^B(p&-1, p) = \left\{ \mu \in  \mathbb Z \ | \ (\mu | \epsilon_i-\epsilon_{i+1}) \in \Z_{\geq 0}, \ i =1, \dots, n-1, \  (\mu | 2\epsilon_n) \in \Z_{\geq 0}, \frac{}{} \ (\mu | \epsilon_1+\epsilon_2) \leq p-2n \right\}, \\
P^B_Q(p&-1, p) = \left\{ \mu \in  \mathbb Z \ | \ (\mu | \epsilon_i-\epsilon_{i+1}) \in \Z_{\geq 0}, \ i =1, \dots, n-1, \  (\mu | \epsilon_n) \in \Z_{\geq 0}, \frac{}{} \ (\mu | \epsilon_1+\epsilon_2) \leq p-2n \right\}. \\
 \end{split}
\end{equation}
We observe that the map $\lambda \mapsto \sqrt{2}\lambda$ provides the isomorphism
\begin{equation}\label{eq:v=2}
P^C(p, 2) \cong P^B_Q(p-1, p).
\end{equation}

Let now $\g = \mathfrak{osp}_{1|2n}$. This is not a Lie algebra, but a Lie superalgebra. Its simple roots can be identified with the simple roots of $B_n$ where the short roots are odd roots and the long ones are even. Note that $2\alpha$ for $\alpha$ an odd root is an even root and it can be identified with a long root of $C_n$. Weights of $ \mathfrak{osp}_{1|2n}$ can be thus identified with weights of $B_n$ and it turns out that there is a one-to-one correspondence between irreducible finite-dimensional $\mathfrak{osp}_{1|2n}$-modules and irreducible finite-dimensional non-spinor representations of $\mathfrak{so}_{2n+1}$ \cite{RS}. They are both parameterized by weights in $R = Q \cap P^+$, the set of dominant weights that lie in the root lattice of $B_n$. 
\begin{remark}\label{rmk:osp-sp}
Let  $E_\mu$ be the simple highest-weight module of $\mathfrak{osp}_{1|2n}$ of highest weight $\mu$ and let $v_\mu$ be the highest-weight vector. Then $E_\mu$ contains the $\mathfrak{sp}_{2n}$ 
simple highest-weight module $E_{\nu}$ with $\sqrt{2}\nu =\mu$ as submodule and this submodule is generated by $v_\mu$. 

Similarly let $\widetilde V^k(\mu)$ be a $V^k(\mathfrak{osp}_{1|2n})$-module whose top level is isomorphic to $E_\mu$. Then the highest-weight vector generates a $V^k(\mathfrak{sp}_{2n})$-module whose top level is $E_{\nu}$.
\end{remark}

\subsection{Fusion categories of $\W$-algebras}

Let $V^k(\lambda)$ denote the universal  Weyl module of $V^k(\g)$ whose top level is the irreducible highest-weight representation $E_\lambda$ of $\g$ of highest weight $\lambda$. Let $\AL_k(\lambda)$ be its unique graded quotient. 
Let $\W^k(\g)$ the principal $\W$-algebra of $\g$ at level $k$. It is obtained via quantum Hamiltonian reduction from $V^k(\g)$ and the reduction functor is denoted by $H^0_{\text{DS}}( \, \bullet\,)$. Then the reduction $H^0_{\text{DS}}(V^k(\lambda))$ of $V^k(\lambda)$ is a  $\W^k(\g)$-module. 
More generally, there is the  twisted quantum Hamiltonian reduction of Arakawa and Frenkel \cite{AF19}.  These reductions are labelled by elements $\mu$ in the set of dominant coweights, ${}^L{P}^+$ and one denotes them as
\[
T^{\kappa}_{\lambda, \mu} =  H^0_{\text{DS}, \mu}(V^k(\lambda))\qquad \kappa = k+ h^\vee.
\]
One has $T^{\kappa}_{\lambda, \mu} = {}^L{T}^{{}^L{\kappa}}_{\mu, \lambda}$ by \cite[Thm. 2.2]{AF19}, where ${}^L\kappa = {}^Lk +{}^Lh^\vee$.

Let $\W_k(\g)$ be the unique simple quotient of $\W^k(\g)$ and denote the simple quotient of $T^{\kappa}_{\lambda, \mu}$ by $\WL_k(\lambda, \mu)$
Let $k$ be non-degenerate admissible. If $k$ is principal or coprincipal non-degenerate admissible, then $\W_k(\g)$ is strongly rational \cite{Ara-rat} and the simple objects are in the principal admissible case the $\WL_k(\lambda, \mu)$ with $(\lambda, \mu) \in P(p, q) \times {}^LP(q, r^\vee p) \cong P(p, q) \times P^\vee(q, r^\vee p) $ and in the coprincipal case the $\WL_k(\lambda, \mu)$ with $(\lambda, \mu) \in P(p, q) \times {}^LP(q/r^\vee,  p) \cong P(p, q) \times P^\vee(q/r^\vee,  p)$. Call a non-degenerate admissible level of coboundary type if $q= h$ in the principal admissible case or $q = r^\vee {}^Lh^\vee$ in the coprincipal admissible case. If the non-degenerate admissible level is not coboundary then
 two simple modules are isomorphic if and only if they are in the same orbit under a certain diagonal Weyl group action \cite{FKW}, in particular $\WL_k(\lambda, 0) \cong \WL_k(\lambda', 0)$ if and only if $\lambda = \lambda'$. The modular transformations of characters have been computed in \cite{AE-mod} and been used to compute fusion rules in \cite{C-fus}. 
\begin{theorem}\textup{\cite{C-fus}}
Let 
\begin{equation}\label{eq:wzwfusion}
\AL_{\ell-h^\vee}(\lambda) \boxtimes \AL_{\ell-h^\vee}(\nu) \cong \bigoplus_{\phi \in P(\ell, 1)} N_{\lambda, \nu}^{\g_\ell  \  \phi}\   \AL_{\ell-h^\vee}(\phi).
\end{equation}
be the fusion rules of $L_{\ell- h^\vee}(\g)$ for $\ell - h^\vee \in Z_{>0}$. Let $k= -h^\vee + \frac{p}{q}$ be principal non-degenerate admissible. Then 
simple modules are parameterized by a quotient of the subset $P(p, 1) \times P^\vee(q, r^\vee)$ of the set $P(p, 1) \times P(q, 1)$ by \eqref{eq:embweights}.
Set $ \WL_k(\lambda, \lambda') = 0$ for $(\lambda, \lambda') \in P(p, 1) \times P(q, 1) \setminus P(p, 1) \times P^\vee(q, r^\vee)$.
With this parameterization the  
 fusion rules are  for $(\lambda, \lambda'), (\nu, \nu') \in P(p, 1) \times P^\vee(q, r^\vee)$
\begin{equation}
\begin{split}
 \WL_k(\lambda, \lambda')  \boxtimes  \WL_{k}(\nu, \nu')  &\cong \bigoplus_{\substack {\phi \in P(p, 1) \\ \phi' \in P(q, 1)}}  N_{\lambda, \nu}^{\g_p \  \phi}  N_{\lambda', \nu'}^{\g_q \  \phi'}   \WL_{k}(\phi, \phi').
\end{split}
\end{equation}
\end{theorem}
The proof of this Theorem used a variant of Verlinde's formula in modular tensor categories, namely that there is an isomorphism given by open Hopf links  between the Grothendieck ring of the category and the endomorphism ring of the direct sum of all inequivalent modules. The open Hopf links are given by normalized modular $S$-matrix coefficients. Let $\zeta = e^{2\pi i /p}$, then the map $\zeta \mapsto \zeta^q$, maps the open Hopf links $\AL_{p-h^\vee}(\lambda)$ of the affine vertex algebra $L_{p-h^\vee}(\g)$ at level $p-h^\vee$ to the corresponding ones, that is $\WL_k(\lambda, 0)$, of the $\W$-algebra $\W_k(\g)$. This map provides a homomorphism of rings. The cokernel of this ring homomorphism has not been discussed in \cite{C-fus}, but it is known:  Note that these Hopf links are given by $q$-characters and in particular coincide with the ones of the corresponding quantum group $U_q(\g)$ for $q = e^{\frac{2\pi i}{2 r^\vee (k + h^\vee)}}$, see e.g. the proof of Theorem 3.3.9 of \cite{BK}. For $k$ non-degenerate principal or coprincipal admissible, the Hopf links vanish for all negligible objects and the non-negligible simple objects are exactly the highest-weight modules of admissible weight at that level, see e.g. Theorem 2 together with Lemma 7 of \cite{Sawin}.

\begin{remark}\label{rmk:prop}
Let $\mathcal C(p, q)$ be the category of $\W_k(\g)$-modules at the non-degenerate  admissible level $k = -h^\vee + \frac{p}{q}$. Let $\mathcal C_L(p, q)$ be the subcategory whose simple objects are isomorphic to modules of type $\WL_k(\lambda, 0)$. Similarly we denote by $\mathcal C_R(p, q)$ be the subcategory whose simple objects are isomorphic to modules of type $\WL_k(0, \lambda')$. We note the properties, see \cite[Theorem 6.1]{C-fus}
\begin{enumerate}
\item The categories $\mathcal C_L(p, q)$ and $\mathcal C_R(p, q)$ are fusion subcategories of $\mathcal C(p, q)$.
\item Let $k = -h^\vee + \frac{p}{v}$ be non-degenerate admissible with $(q, r^\vee) = (v, r^\vee)$, then there is the isomorphism $K[\mathcal C_L(p, q)] \cong K[\mathcal C_L(p, v)]$  of Grothendieck rings if both levels are not coboundary. 
\item $\WL_k(0, \lambda')$ is in the centralizer of $\mathcal C_L(p, q)$ if $\lambda' \in Q$.
\item $\WL_k(\lambda, 0)$ is in the centralizer of $\mathcal C_R(p, q)$ if $\lambda \in Q$.
\end{enumerate}
\end{remark}

\section{The universal affine vertex superalgebra $V^k(\mathfrak{osp}_{1|2n})$}

We recall the notion of equivariant $\W$-algebras using \cite{Arakawa:2018egx}. Let $G$ be an algebraic group and $\kappa \in \C$ generic. The CDO algebra at level $\kappa$, $\CDO_{G, \kappa}$ is a vertex algebra that has an action  of $V^{k}(\g) \otimes V^{\ell}(\g)$ with $\g$ the Lie algebra of $G$ and $\kappa = k+ h^\vee$, $-\kappa = \ell + h^\vee$. It satisfies
\[
\CDO_{G, \kappa} \cong \bigoplus_{\lambda \in P^+} V^k(\lambda) \otimes V^\ell(\lambda^*)
\]
as a vector space graded by conformal weights and weights of $V^{k}(\g) \otimes V^{\ell}(\g)$. Moreover it is a direct sum of both $V^{k}(\g)$ and $V^{\ell}(\g)$-modules in the categories $\KL^k(\g)$ and $\KL^\ell(\g)$ \cite[Prop. 5.6]{Arakawa:2018egx}. Since these categories are completely reducible for non-rational $k, \ell$, it follows that 
\begin{equation}
\CDO_{G, \kappa} \cong \bigoplus_{\lambda \in P^+} V^k(\lambda) \otimes V^\ell(\lambda^*) \qquad\qquad (\kappa \ \text{is generic})
\end{equation}
as $V^{k}(\g) \otimes V^{\ell}(\g)$ for non-rational $k, \ell$. 
The equivariant $\W$-algebra of $G$ at level $\kappa$ with respect to the nilpotent element $f$ in $\g$ is the quantum Hamiltonian reduction corresponding to $f$ on the first factor. In particular
\begin{equation}
\W_{G, f, \kappa} \cong \bigoplus_{\lambda \in P^+} T^{\kappa, f}_{\lambda, 0} \otimes V^\ell(\lambda^*), \qquad T^{\kappa, f}_{\lambda, 0} = H^0_{\text{DS}, f}(V^k(\lambda)), \qquad\qquad (\kappa \ \text{is generic}).
\end{equation}
We restrict to  principal nilpotent elements and will omit the symbol $f$. The $T^{\kappa}_{\lambda, 0}$ are modules for $\W^k(\g)$. 
$\W_{G, \kappa}$ is a strict chiralization of a smooth symplectic variety \cite{Arakawa:2018egx} and hence simple by \cite[Cor. 9.3]{arakawa2019arc}.

\begin{theorem}\label{CDOdecomposition}
Let $F^{2n+1}$ be the vertex superalgebra of $2n+1$ free fermions  and consider $G=Sp(2n)$ and $\g= \mathfrak{sp}_{2n}$. Let $\tau, t$ and $\ell$ be related by $\frac{1}{\tau} + \frac{1}{\ell+h^\vee}=2$, $\ell + t = -2 h^\vee_{\mathfrak{osp}_{1|2n}}$, and let $t$ be generic. There is an embedding $\W^t(\mathfrak{osp}_{1|2n}) \hookrightarrow \W_{Sp(2n), \kappa} \otimes F^{2n+1}$, and $$C^\ell = \text{Com}(\W^t(\mathfrak{osp}_{1|2n}), \W_{Sp(2n), \kappa} \otimes F^{2n+1} )$$ is a simple vertex superalgebra. Moreover, its even and odd parts decompose as $\W^{\tau-h^\vee}(\mathfrak{sp}_{2n}) \otimes V^\ell(\mathfrak{sp}_{2n})$-modules as follows:
\begin{equation}
\begin{split}
C^\ell_{\text{even}}
 &\cong \bigoplus_{\mu \in P^+\cap Q}  T^{\tau}_{\mu, 0} \otimes  V^\ell(\mu), \\
C^\ell_{\text{odd}} &\cong  \bigoplus_{\mu \in P^+\cap P \setminus Q} T^{\tau}_{\mu, 0}  \otimes V^\ell(\mu).  
\end{split}
\end{equation}
Let  $\lambda \in \check{P}^+$ and $C^\ell_\lambda = C^\ell_{\lambda, \text{even}} \oplus C^\ell_{\lambda, \text{odd}}$, such that  
\begin{equation}
\begin{split}
C^\ell_{\lambda, \text{even}}
 &\cong \bigoplus_{\mu \in P^+\cap Q}  T^{\tau}_{\mu, \lambda} \otimes  V^\ell(\mu), \\
C^\ell_{\lambda, \text{odd}} &\cong  \bigoplus_{\mu \in P^+\cap P \setminus Q} T^{\tau}_{\mu, \lambda}  \otimes V^\ell(\mu).  
\end{split}
\end{equation}
Then the $C^\ell_{\lambda}$ are $C^\ell$-modules and 
\begin{equation}
\W_{Sp(2n), \kappa} \otimes F^{2n+1}  \cong \bigoplus_{\lambda \in \check{P}^+\cap \check{Q}} C^\ell_{\lambda} \otimes N^t_{\lambda, 0},
\end{equation}
with $N^t_{\lambda, 0}$ certain inequivalent simple $\W^t(\mathfrak{osp}_{1|2n})$-modules. 
\end{theorem}
\begin{proof}
We recall the twisted quantum Hamiltonian reduction of Arakawa and Frenkel \cite{AF19} for principal $f$.  These reductions are labelled by elements $\mu$ in the set of dominant coweights, $\check{P}^+$ and one denotes 
\[
T^{\kappa}_{\lambda, \mu} =  H^0_{\text{DS}, \mu}(V^k(\lambda)).
\]
One has $T^{\kappa}_{\lambda, \mu} = \check{T}^{\check{\kappa}}_{\mu, \lambda}$ by \cite[Thm. 2,2]{AF19} with the right hand side reductions of modules of $V^{\ell}({}^L\g)$ with ${}^L\g$ the Langlands dual Lie algebra of $\g$ and $\check{\kappa} =\ell + {}^Lh^\vee$ the shifted level and $r^\vee\check{\kappa}\kappa = 1$ the lacity. In particular if $\g = \mathfrak{sp}_{2n}$, then ${}^L\g = \mathfrak{so}_{2n+1}$ and $h^\vee = n+1$, ${}^Lh^\vee = 2n-1$ and $r^\vee =2$. 
The twisted reduction commutes with tensoring with integrable representations in the following sense, 
\begin{equation}\label{translation}
H^0_{\text{DS}, \mu}(V^k(\lambda) \otimes  L) \cong H^0_{\text{DS}, \mu}(V^k(\lambda)) \otimes \sigma^*_\mu L
\end{equation}
as  $\W^k(\g) \otimes L$-modules \cite[Thm. 7.2]{ACF}, where $L$ is an integrable module (at level $\ell \in \mathbb Z_{>0}$) for $\g$ and the left-hand side is the diagonal reduction at level $k+n$. 
$ \sigma^*_\mu L$ is the spectrally flown module of $L$. This is defined in terms of Li's $\Delta$-operator \cite{Li} and it is again an integrable module at the same level. It is the module twisted by the automorphism induced from the Weyl translation
\[
t_\mu: \lambda \mapsto \lambda + \lambda(K)\mu  - \left( (\lambda | \mu) +\frac{( \mu | \mu )}{2}\lambda(K)\right) \delta. 
\]
If $L=V_N$ is a lattice VOA, then one has $ \sigma^*_\mu (V_N) \cong V_{N+\mu}$ and more generally for a module $V_{N+\nu}$ of $V_N$ one has $ \sigma^*_\mu (V_{N+\nu}) \cong V_{N+\nu+\mu}$. In particular one has $ \sigma^*_\mu V_N \cong V_{N}$ if $\mu \in N$. Let $N = \mathbb Z^n$, so that $V_N$ is the vertex superalgebra of $n$-pairs of free fermions. Then $N$ coincides with the weight lattice $P$ of $\mathfrak{sp}_{2n}$, that is let $\epsilon_1, \dots, \epsilon_n$ be an orthonormal basis of $N$ and set $\alpha_i = \epsilon_1 - \epsilon_{i+1}$ for $i =1, \dots, n-1$ and $\alpha_n = \epsilon_n$. 
Then $\alpha_1, \dots, \alpha_{n-1}, 2\alpha_n$ span a sublattice of $N$, that coincides with the root lattice $Q$ of $\mathfrak{sp}_{2n}$. $V_N$ is a subalgebra of $F^{2n+1}$, the vertex superalgebra of $2n+1$ free fermions and $F^{2n+1} \cong L_1(\mathfrak{so}_{2n+1}) \oplus L_1(\omega_1)$ is an integrable module for $\mathfrak{so}_{2n+1}$ at level one. The even part is  $L_1(\mathfrak{so}_{2n+1})$ and the odd one $L_1(\omega_1)$. We thus see that 
\begin{equation}
\begin{split}
\sigma^*_\mu(L_1(\mathfrak{so}_{2n+1})) &\cong  \begin{cases} L_1(\mathfrak{so}_{2n+1}) & \quad \text{if} \ \mu \in Q \\  L_1(\omega_1)& \quad \text{if} \ \mu \in P\setminus Q \end{cases}, \\
\sigma^*_\mu(L_1(\omega_1)) &\cong  \begin{cases}L_1(\omega_1)  & \quad \text{if} \ \mu \in Q \\  L_1(\mathfrak{so}_{2n+1}) & \quad \text{if} \ \mu \in P\setminus Q \end{cases},\\
\sigma^*_\mu(F^{2n+1}) &\cong  \begin{cases} F^{2n+1} & \quad \text{if} \ \mu \in Q \\  {}^\pi F^{2n+1} & \quad \text{if} \ \mu \in P\setminus Q. \end{cases}
\end{split}
\end{equation}
with ${}^\pi F^{2n+1}$ the parity reverse of $F^{2n+1}$. Recall that $R= Q \cap P^+$ denotes the set of highest weights of tensor representations of $\mathfrak{so}_{2n+1}$, i.e. of modules appearing as direct summands of the iterated tensor product of the standard representation. 
Let $k$ be generic, then by \cite[Theorem 4.1]{CL21} 
\begin{equation}\label{eq:B-coset} 
V^k(\mathfrak{so}_{2n+1}) \otimes F^{2n+1} \cong \bigoplus_{\lambda \in R} V^{k+1}(\lambda) \otimes N^t_{\lambda, 0}
\end{equation}
with  $N^t_{\lambda, 0}$ certain modules of $\W^t(\mathfrak{osp}_{1|2n})$ at level $t$ determined by 
\[
2t+ 2n+1 = \frac{k+2n}{k+2n-1}.
\]
The decomposition \eqref{eq:B-coset} is completely reducible by \cite[Theorem 4.12]{CL20}.
Applying $H^0_{\text{DS}, \mu}$ to \eqref{eq:B-coset} and using \eqref{translation} and $T^{\kappa}_{\lambda, \mu} = \check{T}^{\check{\kappa}}_{\mu, \lambda}$ one gets
\begin{equation}
\begin{split}
T^\kappa_{\mu, 0} \otimes F^{2n+1} &\cong T^{\check{\kappa}}_{0, \mu} \otimes F^{2n+1} 
\cong H^0_{\text{DS}, \mu}(V^k(\mathfrak{so}_{2n+1})) \otimes F^{2n+1} \\
&\cong H^0_{\text{DS}, \mu}(V^k(\mathfrak{so}_{2n+1}) \otimes F^{2n+1}) 
\cong \bigoplus_{\lambda \in R} H^0_{\text{DS}, \mu}(V^{k+1}(\lambda)) \otimes N^t_{\lambda, 0} \\
&\cong \bigoplus_{\lambda \in R} T^{\check{\kappa}+1}_{\lambda, \mu} \otimes N^t_{\lambda, 0}
\cong \bigoplus_{\lambda \in R} T^{\tau}_{\mu, \lambda} \otimes N^t_{\lambda, 0}.
\end{split}
\end{equation}
with $\check{\kappa} = k +2n-1$ and $\tau$ dual to $\check{\kappa}+1$, that is $2\tau (\check{\kappa} +1) =1$. This equality holds for $\mu \in Q$, while for $\mu \in P\setminus Q$ one has to add an additional parity reversal. 
\begin{equation}
\begin{split}
T^\kappa_{\mu, 0} \otimes F^{2n+1} &\cong T^{\check{\kappa}}_{0, \mu} \otimes F^{2n+1} 
\cong H^0_{\text{DS}, \mu}(V^k(\mathfrak{so}_{2n+1})) \otimes F^{2n+1} \\
&\cong H^0_{\text{DS}, \mu}(V^k(\mathfrak{so}_{2n+1}) \otimes {}^\pi F^{2n+1}) 
\cong \bigoplus_{\lambda \in R} H^0_{\text{DS}, \mu}(V^{k+1}(\lambda)) \otimes {}^\pi N^t_{\lambda, 0} \\
&\cong \bigoplus_{\lambda \in R} T^{\check{\kappa}+1}_{\lambda, \mu} \otimes {}^\pi N^t_{\lambda, 0}
\cong \bigoplus_{\lambda \in R} T^{\tau}_{\mu, \lambda} \otimes {}^\pi N^t_{\lambda, 0}.
\end{split}
\end{equation}
It thus follows that (note that $V^\ell(\mu^*) \cong V^\ell(\mu)$ for $\g=\mathfrak{sp}_{2n}$),  
\begin{equation}
\begin{split}
\W_{Sp(2n), \kappa} \otimes F^{2n+1} &\cong \bigoplus_{\mu \in P^+} T^{\kappa}_{\mu, 0} \otimes V^\ell(\mu) \otimes F^{2n+1} \\
&\cong \bigoplus_{\mu \in P^+\cap Q} \bigoplus_{\lambda \in R} T^{\tau}_{\mu, \lambda} \otimes N^t_{\lambda, 0} \otimes V^\ell(\mu) \  \oplus \
 \bigoplus_{\mu \in P^+\cap P \setminus Q} \bigoplus_{\lambda \in R} T^{\tau}_{\mu, \lambda} \otimes {}^\pi N^t_{\lambda, 0} \otimes V^\ell(\mu)  \\
\end{split}
\end{equation}
We thus obtain that 
\begin{equation}
\begin{split}
C^\ell &:= \text{Com}(\W^t(\mathfrak{osp}_{1|2n}), \W_{Sp(2n), \kappa} \otimes F^{2n+1} )
 \cong \bigoplus_{\mu \in P^+\cap Q}  T^{\tau}_{\mu, 0} \otimes  V^\ell(\mu) \ \ \oplus\ \ 
 \bigoplus_{\mu \in P^+\cap P \setminus Q} {}^\pi(T^{\tau}_{\mu, 0}  \otimes V^\ell(\mu)). 
\end{split}
\end{equation}
Since   $\W_{Sp(2n), \kappa} \otimes F^{2n+1}$ is a simple vertex superalgebra on which $\W^t(\mathfrak{osp}_{1|2n})$ acts completely reducibly, we have that $\text{Com}(\W^t(\mathfrak{osp}_{1|2n}), \W_{Sp(2n), \kappa} \otimes F^{2n+1} )$ is simple \cite[Proposition 5.4]{CGN}. Recall that $2\tau (\check{\kappa}+1) = 1 = 2\kappa \check{\kappa}$ and $\ell + h^\vee = -\kappa$. Hence
\[
\frac{1}{\tau} + \frac{1}{\ell + h^\vee} = 2(\check{\kappa}+1)  -\frac{1}{\kappa} = 2 +\frac{1}{\kappa} - \frac{1}{\kappa} = 2
\] 
and $t2+t\ell + 4n+2 = \frac{\check\kappa+1}{\check \kappa} -\kappa -1 =0$.
\end{proof}
\begin{theorem}\label{thm1}
Let $\ell$ be generic, then $C^\ell \cong V^\ell(\mathfrak{osp}_{1|2n})$ and $C^\ell_\lambda \cong V^\ell(\lambda)$. 
\end{theorem}
\begin{proof}
$C^\ell$ and $V^\ell(\mathfrak{osp}_{1|2n})$ are simple vertex operator superalgebras whose graded characters coincide by Theorems \ref{thm:character} and \ref{CDOdecomposition}. In particular their weight one subspaces coincide as $\mathfrak{sp}_{2n}$-modules. The weight one subspace of $C^\ell$ must generate an affine vertex superalgebra and the only possibility for this is to be of type $\mathfrak{osp}_{1|2n}$ at level $\ell$. Thus  $C^\ell \cong V^\ell(\mathfrak{osp}_{1|2n})$ must hold. The $C^\ell$ are $V^\ell(\mathfrak{osp}_{1|2n})$-modules whose graded character coincides with the one of $V^\ell(\lambda)$ by Theorems \ref{thm:character} and \ref{CDOdecomposition}. In particular their top levels coincide as $\mathfrak{sp}_{2n}$-modules and hence as $\mathfrak{osp}_{1|2n}$-modules. This is only possible if $C^\ell_\lambda \cong V^\ell(\lambda)$.
\end{proof}

\section{The simple affine vertex superalgebra $L_\ell(\mathfrak{osp}_{1|2n})$}\label{sec:adm}

We now want to pass from the generic levels to interesting rational levels. 
\begin{remark}\label{remark:decomposition}
In general if we have a family of vertex operator superalgebras $A^k$, that for generic level decomposes as a direct sum of modules for a commuting pair of subalgebras $V^k, W^k$
\[
A^k = \bigoplus_{\lambda \in I} V^k_\lambda \otimes W^k_\lambda, 
\]
then at any specific level $\ell$ this decomposition might not be completely reducible, as summands might have non-trivial submodules and there might be extensions between different modules appearing at this specific level. However any simple composition factor of $A^\ell$ must be a composition factor of $V^\ell_\lambda \otimes W^\ell_\lambda$ for some $\lambda \in I$. The same statement applies of course for the simple quotient $A_\ell$ of $A^\ell$. 
Let $V_\ell, W_\ell$ be the simple quotients of $V^\ell, W^\ell$. If we assume that
\begin{enumerate}
\item $A^k$ is an ordinary module for $V^k \otimes W^k$.
\item The categories of ordinary modules of $V_\ell$ and  $W_\ell$ are both semisimple.
\item  Let $I^V_\ell$ denote the set of simple objects of $V_\ell$, then for $\lambda \in I$, $L$ in  $I^V_\ell$  is a composition factor of $V^\ell_\lambda$ implies $L$ is the simple quotient of $V^\ell_\lambda$, $L \cong L^\ell_\lambda$.
\item  Let $I^W_\ell$ denote the set of simple objects of $W_\ell$, then for $\lambda \in I$, $M$ in  $I^W_\ell$  is a composition factor of $W^\ell_\lambda$ implies $M$ is the simple quotient of $V^\ell_\lambda$, $M \cong M^\ell_\lambda$.
\end{enumerate}
Then the only possibility for $A_\ell$ is to be of the form
\[
A_\ell = \bigoplus_{\lambda \in J} L^\ell_\lambda \otimes M^\ell_\lambda \otimes M(\lambda). 
\]
with $J \subset I$ the subset of those weights $\lambda$ with the property that $ L^\ell_\lambda$ in $I^V_\ell$ and $ M^\ell_\lambda$ in $I^W_\ell$. Here $M(\lambda) \in \{0, \mathbb C\}$ is a multiplicity space. The same reasoning holds for ordinary modules of $A^k$. 

Let us now assume that $V^k$ is an affine vertex algebra and $W^k$ is a principal $W$-algebra. Assume that $V_\ell$ is at admissible level and $W_k$ is at non-degenerate admissible level. Then (2) of the assumptions holds by the main Theorems of \cite{Ara-adm,Ara-rat}, assumption (3) holds by Theorem 0.5 of \cite{GK11} and assumption (4) by Corollary 10.9 of \cite{Ara-rat} together with Theorem 0.5 of \cite{GK11}.
\end{remark}

Let $\g =\mathfrak{sp}_{2n}$ and $\ell = -h^\vee + \frac{p}{q}$ be non-degenerate admissible. We prefer to parameterize modules by weight labels $(\lambda, \mu)$ with $\lambda$ a weight of $\mathfrak{sp}_{2n}$ and $\mu$ one of $\mathfrak{so}_{2n+1}$, i.e. the simple modules are denoted by $\WL_\ell(\lambda, \mu)$ with $\lambda \in P^C(p, q)$ and $\mu \in P^B(q, r^\vee p)$ if $\ell$ is principal admissible and $\mu \in P^B(q/r^\vee, p)$ if $\ell$ is coprincipal admissible.

\begin{theorem}\label{thm2}
Let $k = -h^\vee + \frac{u}{v}$ be admissible and $\ell = -h^\vee + \frac{u}{2u-v}$ non-degenerate admissible for $\mathfrak{sp}_{2n}$. 
Let $P^B := P^B(q, r^\vee p)$ if $\ell$ is principal admissible and $P^B:= P^B(q/r^\vee, p)$ if $\ell$ is coprincipal admissible and $P^B_Q := P^B \cap Q$. 
Then 
\begin{enumerate}
\item $L_k(\mathfrak{osp}_{1|2n})$ is  a  $L_k(\mathfrak{sp}_{2n}) \otimes \W_\ell(\mathfrak{sp}_{2n})$-module. 
\item 
\[
L_k(\mathfrak{osp}_{1|2n}) \cong \bigoplus_{\lambda \in P(u, v) } \AL_k(\lambda) \otimes \WL_\ell(\lambda, 0)
\]
and for $\mu \in P_Q^B$
\[
\AL_k(\mu) \cong \bigoplus_{\lambda \in P(u, v) } \AL_k(\lambda) \otimes \WL_\ell(\lambda, \mu)
\]
as $L_k(\mathfrak{sp}_{2n}) \otimes \W_\ell(\mathfrak{sp}_{2n})$-modules.
\item The  $\AL_k(\mu)$  with $\mu \in P_Q^B$ are a complete list of inequivalent simple ordinary $L_k(\mathfrak{osp}_{1|2n})$-modules.
\item The category of ordinary $L_k(\mathfrak{osp}_{1|2n})$-modules is a semisimple, rigid vertex tensor supercategory, i.e. it is a fusion supercategory.
\end{enumerate}
\end{theorem}
\begin{proof} ${}$

\noindent $(1) $ 
Let $\widetilde L_{k}(\mathfrak{sp}_{2n})$ be the image of $V^k(\mathfrak{sp}_{2n})$ in $ L_k(\mathfrak{osp}_{1|2n})$.
By \cite[Theorem 8.1]{CL-cosets} the coset $\text{Com}(\widetilde L_{k}(\mathfrak{sp}_{2n}), L_k(\mathfrak{osp}_{1|2n}))$ is a homomorphic image of $\W^\ell(\mathfrak{sp}_{2n})$. By \cite[Theorem 4.1]{ACK21} this coset is simple, i.e. 
$\text{Com}(\widetilde L_{k}(\mathfrak{sp}_{2n}), L_k(\mathfrak{osp}_{1|2n})) \cong \W_\ell(\mathfrak{sp}_{2n})$. Since $ \W_\ell(\mathfrak{sp}_{2n})$ is strongly rational, $\text{Com}( \W_\ell(\mathfrak{sp}_{2n}), L_k(\mathfrak{osp}_{1|2n}))$ is simple by \cite[Lemma 2.1]{ACKL}. It follows that 
$\widetilde L_{k}(\mathfrak{sp}_{2n})\cong L_{k}(\mathfrak{sp}_{2n})$ by \cite[Theorem 3.4]{AEM}.
Hence
$L_k(\mathfrak{osp}_{1|2n})$ is a  $L_k(\mathfrak{sp}_{2n}) \otimes \W_\ell(\mathfrak{sp}_{2n})$-module. 

\noindent $(2)$ for $v =1, 2$.
Let $\cC$ be the category of ordinary $L_k(\mathfrak{sp}_{2n}) \otimes \W_\ell(\mathfrak{sp}_{2n})$-modules. The category of ordinary $L_k(\mathfrak{sp}_{2n})$-modules is semisimple \cite{Ara-adm} and a vertex tensor category \cite{CHY} and since the category of $\W_\ell(\mathfrak{sp}_{2n})$-modules is a modular tensor category \cite{H-mod, H-ver}
as $\W_\ell(\mathfrak{sp}_{2n})$ is rational \cite{Ara-rat}. Thus Theorem 5.5 of \cite{CKM2} applies, that is $\cC$ is the Deligne product of the category of ordinary modules of 
$L_k(\mathfrak{sp}_{2n})$ and the category of $\W_\ell(\mathfrak{sp}_{2n})$-modules. 
It follows that $L_k(\mathfrak{osp}_{1|2n})$ is a commutative superalgebra object, call it $A$, in $\cC$ by \cite{CKL}. 
Hence the category of ordinary 
$L_k(\mathfrak{osp}_{1|2n})$ -modules is a vertex tensor category and equivalent to the category $\cCloc$ of local $A$-modules in $\cC$ \cite{CKM1}.

Since $A =L_k(\mathfrak{osp}_{1|2n})$ is a homomorphic image of $V^k(\mathfrak{osp}_{1|2n})$  it must  be of the form (recall Remark \ref{remark:decomposition})
\[
A =  L_k(\mathfrak{osp}_{1|2n}) \cong \bigoplus_{\lambda \in P(u, v) } \AL_k(\lambda) \otimes \WL_\ell(\lambda, 0) \otimes M(\lambda)
\]
with $M(\lambda) \in \{ 0, \C\}$. We have to show that $M(\lambda) = \C$. 
The subcategory of ordinary $L_k(\mathfrak{sp}_{2n})$-modules whose simple objects are the $\AL_k(\lambda)$ with  $M(\lambda) = \C$ is a rigid tensor subcategory of ordinary $L_k(\mathfrak{sp}_{2n})$-modules by Theorem \ref{thm:rigid}.
Let $\cD$ be the corrsponding subcategory of $\cC$, i.e. simple objects are simple $L_k(\mathfrak{sp}_{2n})$-modules in this rigid subcategory times simple 
$\W_\ell(\mathfrak{sp}_{2n})$-modules.
In particular $A =  L_k(\mathfrak{osp}_{1|2n})$ is an object in $\cD$ and we can consider the category of local $A$-modules in $\cD$, which we denote by $\cDloc$. 
The category $\cDloc$ is semisimple by Theorem \ref{thm:semisimple}. $\cDloc$ is a subcategory of $\cCloc$ that is closed under submodules; in particular, a simple object in $\cDloc$ is also simple in $\cCloc$. 
Let $\cF$ be the induction functor from $\cD$ to $\cDloc$. 
Note that 
\[
X_\mu = A \boxtimes (L_k(\mathfrak{sp}_{2n})) \otimes \WL_\ell(0, \nu)))  = \bigoplus_{\lambda \in P(u, v) } \AL_k(\lambda) \otimes \WL_\ell(\lambda, \mu) \otimes M(\lambda)
\]
as object in $\cC$. 
Let $P^B := P^B(q, r^\vee p)$ if $\ell$ is principal admissible and $P^B:= P^B(q/r^\vee, p)$ if $\ell$ is coprincipal admissible. The $X_\mu := \cF(L_k(\mathfrak{sp}_{2n})) \otimes \WL_\ell(0, \mu))$ is in $\cDloc$ if $\mu \in P^B_Q$ by Remark \ref{rmk:prop} (3).
Let $P^B_Q:= P^B\cap Q$. By Frobenius reciprocity, Theorem \ref{thm:Frec}, for $\mu, \nu \in P^B_Q$
\begin{equation}\nonumber
\begin{split}
\text{Hom}_{\cDloc}\left( X_\mu, X_\nu \right) &=  
\text{Hom}_{\cDloc}\left(\cF(L_k(\mathfrak{sp}_{2n}) \otimes \WL_\ell(0, \mu)), \cF(L_k(\mathfrak{sp}_{2n}) \otimes \WL_\ell(0, \nu))  \right)  \\
&\cong \text{Hom}_{\cD}\left(L_k(\mathfrak{sp}_{2n})) \otimes \WL_\ell(0, \mu), A \boxtimes (L_k(\mathfrak{sp}_{2n})) \otimes \WL_\ell(0, \nu)))  \right)  \\
&\cong \text{Hom}_{\cD}\left(L_k(\mathfrak{sp}_{2n})) \otimes \WL_\ell(0, \mu), L_k(\mathfrak{sp}_{2n})) \otimes \WL_\ell(0, \nu))  \right) \\
&\cong  \delta_{\mu, \nu} \C. 
\end{split}
\end{equation}
Thus the $X_\mu$ are simple ordinary $A =  L_k(\mathfrak{osp}_{1|2n})$-modules. On the other hand the ordinary module of highest weight $\mu$ must be of the form
\[
\AL_k(\mu) \cong \bigoplus_{\lambda \in P(u, v) } \AL_k(\lambda) \otimes \WL_\ell(\lambda, \mu) \otimes M(\lambda, \mu)
\]
with $M(\lambda, \mu) \in \{ 0, \C\}$. 
and so we can conclude that $M(\lambda, \mu) = M(\lambda)$ and $X_\mu \cong \AL_k(\mu)$.
By Remark \ref{rmk:osp-sp} and semisimplicity of ordinary $L_k(\mathfrak{sp}_{2n})$-modules, the $L_k(\mathfrak{osp}_{1|2n})$-module $\AL_k(\mu)$ viewed as an $L_k(\mathfrak{sp}_{2n})$-module contains the $L_k(\mathfrak{sp}_{2n})$-module $\AL_k(\nu)$ with $\mu = \sqrt{2} \nu$ as submodule. 
It follows that $M(\nu) = \C$ if $\sqrt{2}\nu \in P^B_Q$. Since $P^B_Q \cong P(u, v)$ for $v=1, 2$ by \eqref{eq:v=1} and \eqref{eq:v=2} we have proven our second claim in these two cases. 

\noindent $(2)$ for general $v$. Let $\cO_k(\mathfrak{sp}_{2n})$ be the category of ordinary modules of $L_k(\mathfrak{sp}_{2n})$ and $\cC(\W_\ell(\mathfrak{sp}_{2n}))$ be the category of $\W_\ell(\mathfrak{sp}_{2n})$-modules. $L_k(\mathfrak{osp}_{1|2n})$ is weakly generated by the odd dimension one fields, these correspond to the top level of $\AL_k(\omega_1) \otimes \WL_\ell(\omega_1, 0)$ ($\omega_1$ is the first fundamental weight which is the highest weight of the standard representation of $\mathfrak{sp}_{2n}$). By Lemma \ref{lem:gen} $\cC(\W_\ell(\mathfrak{sp}_{2n}))_L$ is generated by $\WL_\ell(\omega_1, 0)$ for $v = 1, 2$. By Remark \ref{rmk:prop} the Grothendieck ring only depends on the numerator of the shifted level and whether the level is principal or coprincipal non-degenerate admissible. Hence $\cC(\W_\ell(\mathfrak{sp}_{2n}))_L$ is generated by $\WL_\ell(\omega_1, 0)$ for all $v$. 
By Lemma \ref{lem:gen} all simple $\cC(\W_\ell(\mathfrak{sp}_{2n}))_L$-modules appear in the decomposition of $L_k(\mathfrak{osp}_{1|2n})$, i.e. $M(\lambda) = \C$ for all $\lambda \in  P(u, v)$.

\noindent $(3)\ \ $ Let $X$ be a simple ordinary $L_k(\mathfrak{osp}_{1|2n})$-module. Since any indecomposable ordinary module of the universal affine vertex operator superalgebra is a homomorphic image of an universal Weyl module $V^k(\mu)$ with $\mu$ a dominant weight of $\mathfrak{so}_{2n+1}$ that is non-spinor and since $X$ is necessarily a $L_k(\mathfrak{sp}_{2n}) \otimes \W_\ell(\mathfrak{sp}_{2n})$-module it follows that 
\[
 X \cong \bigoplus_{\lambda \in P(u, v) } \AL_k(\lambda) \otimes \WL_\ell(\lambda, \mu) \otimes M(\lambda, \mu)
\]
with $\mu \in P^B_Q$ and we already proved that in this case $X \cong \AL_k(\mu)$.

\noindent $(4) \ \ $ Let $\cC(\W_\ell(\mathfrak{sp}_{2n}))_R^Q$ be the subcategory of $\cC(\W_\ell(\mathfrak{sp}_{2n}))_R$ whose weight labels lie in $Q$ (i.e. they are non-spinor). The functor $\cH: \cC(\W_\ell(\mathfrak{sp}_{2n}))_R^Q \rightarrow \cO_k(\mathfrak{osp}_{1|2n})$ defined by
\begin{equation}
\begin{split}
\cH(X) = \cF(L_k(\mathfrak{sp}_{2n}) \otimes X ), \qquad \cH(f) =  \cF(\text{Id}_{L_k(\mathfrak{sp}_{2n})} \otimes f)
\end{split}
\end{equation}
for $X$ an object and $f$ a morphism in $\cC(\W_\ell(\mathfrak{sp}_{2n}))_R^Q$ is braided monoidal as a composition of braided monoidal functors.
It 
is fully faithful by Frobenius reciprocity and it is essentially surjective, i.e. it is an equivalence. Thus the stated properties of $ \cO_k(\mathfrak{osp}_{1|2n})$ are true as they are true for $\cC(\W_\ell(\mathfrak{sp}_{2n}))_R^Q$.
\end{proof}
\begin{corollary}\label{cor1} Let $k$ be principal or coprincipal admissible for $\mathfrak{sp}_{2n}$. 
\begin{enumerate}
\item $\cO_k(\mathfrak{sp}_{2n})$ is rigid.
\item $\cO_k(\mathfrak{sp}_{2n})$ is generated by $\AL_k(\omega_1)$. 
\item There is a super-braid-reversed equivalence $\cO_k(\mathfrak{sp}_{2n}) \cong \cC(\W_\ell(\mathfrak{sp}_{2n}))_L$ sending $\AL_k(\lambda)$ to $\WL_\ell(\lambda, 0)^*$. 
\item There is a super-braided equivalence $\cC(\W_\ell(\mathfrak{sp}_{2n}))_R^Q \cong  \cO_k(\mathfrak{osp}_{1|2n})$ sending $\WL_\ell(0, \mu)$ to $\AL_k(\mu)$. 
\end{enumerate}
\end{corollary}
\begin{proof}
The first and third satement follow from Theorem \ref{thm:rigid}. The second one from Lemma \ref{lem:gen}. The last equivalence is given by the functor 
sending $\WL_\ell(0, \mu)$ to $\cF(L_k(\mathfrak{sp}_{2n} \otimes \WL_\ell(0, \mu)) \cong L_k(\mu)$ which is essentially surjective by the previous theorem and fully faithful by Frobenius reciprocity, Theorem \ref{thm:Frec}.
\end{proof}

\begin{corollary}\label{cor2}
The category of Ramond twisted modules is semisimple and the simple objects are parameterized by $\mu \in P^B \setminus P^B_Q$ and they decompose as
\[
\AL_k(\mu) \cong \bigoplus_{\lambda \in P(u, v) } \AL_k(\lambda) \otimes \WL_\ell(\lambda, \mu).
\]
\end{corollary}
\begin{proof}
We use the properties stated in Remark \ref{rmk:proptw} in our setting. 

Let $\mu \in P^B \setminus P^B_Q$, then $e^{2\pi i \mu \lambda}$ is one for $\lambda \in  Q$ and minus one otherwise. This means that the monodromy of $L_k(\mathfrak{sp}_{2n}) \otimes \WL_\ell(0, \mu)$ with the even part of $L_k(\mathfrak{osp}_{1|2n})$ is one and with the odd part it is minus one. This is precisely the condition that $\AL_k(\mu) := \cF(L_k(\mathfrak{sp}_{2n}) \otimes \WL_\ell(0, \mu))$  is a Ramond twisted module.

 Let $\cO^{\text{R}}_k(\mathfrak{osp}_{1|2n})$ be the category whose simple objects are such $\AL_k(\mu)$.
Let  $\cC(\W_\ell(\mathfrak{sp}_{2n}))_R^{P\setminus Q}$ be the subcategory of $\cC(\W_\ell(\mathfrak{sp}_{2n}))_R$ whose simple objects are the $\WL_\ell(0, \nu))$ with $\nu \in P^B \setminus P^B_Q$. Since monodromy respects tensor products, $\cC(\W_\ell(\mathfrak{sp}_{2n}))_R^{P\setminus Q}$ must be a module category for $\cC(\W_\ell(\mathfrak{sp}_{2n}))_R^Q$ and the tensor product of two objects in  $\cC(\W_\ell(\mathfrak{sp}_{2n}))_R^{P\setminus Q}$ must be in $\cC(\W_\ell(\mathfrak{sp}_{2n}))_R^Q$. 
Since induction is monoidal, it follows that $\cO^{\text{R}}_k(\mathfrak{osp}_{1|2n})$ is a $\cO_k(\mathfrak{osp}_{1|2n})$-module category as well and the tensor product of two objects in $\cO^{\text{R}}_k(\mathfrak{osp}_{1|2n})$  must be local. With $\W_\ell(0, \mu)$ in $\cC(\W_\ell(\mathfrak{sp}_{2n}))_R^{P\setminus Q}$, the same must be true for its dual. Induction preserves duality and hence the dual of any object in $\cO^{\text{R}}_k(\mathfrak{osp}_{1|2n})$ is in $\cO^{\text{R}}_k(\mathfrak{osp}_{1|2n})$ as well. 
 
Consider an arbitrary twisted module $X$ and $Y$ in $\cO^{\text{R}}_k(\mathfrak{osp}_{1|2n})$. Let $Y^*$ be the dual of $Y$, then 
\[
 X \cong L_k(\mathfrak{osp}_{1|2n}) \boxtimes X  \subset (Y^* \boxtimes Y) \boxtimes X \cong Y^* \boxtimes (Y\boxtimes X). 
\] 
Since the tensor product of two twisted modules is local and since $\cO^{\text{R}}_k(\mathfrak{osp}_{1|2n})$ is a $\cO_k(\mathfrak{osp}_{1|2n})$-module category, it follows that $X$ is in $\cO^{\text{R}}_k(\mathfrak{osp}_{1|2n})$.
\end{proof}

\appendix

\section{Admissible levels}

Let $\mathfrak{g}=\mathfrak{osp}_{1|2n}$ with a non-degenerate even supersymmetric invariant bilinear form $(\cdot|\cdot)$, $\mathfrak{h}$ be a Cartan subalgebra of a Lie superalgebra $\mathfrak{osp}_{1|2n}$, $\Delta$ be the root system of $\mathfrak{osp}_{1|2n}$ with respect to $\mathfrak{h}^*$, $\Pi=\{\alpha_1, \ldots, \alpha_n\}$ be a set of simple roots of $\Delta$. Then the highest root of $\mathfrak{osp}_{1|2n}$ is equal to $\theta = 2\alpha_1 + \cdots + 2\alpha_n$. Suppose that $\alpha_n$ is a (unique) non-isotropic odd simple root, the bilinear form $(\cdot|\cdot)$ on $\mathfrak{osp}_{1|2n}$ is normalized as $(\theta|\theta) = 2$, and $\alpha_i$'s satisfy that
\begin{align*}
&(\alpha_i|\alpha_i) = 1,\quad
(\alpha_i|\alpha_{i+1}) = -\frac{1}{2},\quad
i=1, \ldots, n-1,\\
&(\alpha_n|\alpha_n) = \frac{1}{2},\quad
(\alpha_i|\alpha_j) = 0,\quad
|i-j|>1.
\end{align*}
Let $\widetilde{\mathfrak{g}}=\mathfrak{osp}_{1|2n}[t, t^{-1}] \oplus \C K \oplus \C D$ be the (untwisted) affine Lie superalgebra of $\mathfrak{osp}_{1|2n}$ equipped with the following Lie superbrackets:
\begin{align*}
&[a\otimes t^{m_1}, b\otimes t^{m_2}] = [a,b]\otimes t^{m_1+m_2} + m_1(a|b)\delta_{m_1+m_2,0}K,\\
&[D, a\otimes t^{m_1}]=m_1 a\otimes t^{m_1},\quad
[K, \widetilde{\mathfrak{g}}]=0
\end{align*}
for $a, b \in \mathfrak{osp}_{1|2n}$ and $m_1, m_2 \in \Z$. Let $\widetilde{\mathfrak{h}} = \mathfrak{h} \oplus \C K\oplus\C D$ be a Cartan subalgebra of $\widetilde{\mathfrak{g}}$. Then the bilear form on $\mathfrak{h}$ extends to $\widetilde{\mathfrak{h}}$ such that $(K|D)=1$ and $(K|K) = (D|D) = (h|K) = (h|D) =0$ for all $h\in \mathfrak{h}$. Define a lienar isomorphism $\widetilde{\nu} \colon \widetilde{\mathfrak{h}}^* \rightarrow \widetilde{\mathfrak{h}}$ by $(\widetilde{\nu}(\alpha)|h) = \alpha(h)$. Then $\widetilde{\mathfrak{h}}^*$ has a non-degenerate bilinear form by $(\alpha|\beta) = (\widetilde{\nu}(\alpha)|\widetilde{\nu}(\beta))$. Let $\widehat{\Delta}$ be the root system of $\widetilde{\mathfrak{g}}$ with respect to $\widetilde{\mathfrak{h}}^*$ and $\widehat{\Pi} = \{\alpha_0\}\sqcup\Pi$ be a set of simple roots of $\widehat{\Delta}$. Set an imaginary root $\delta:=\alpha_0+\theta$ in $\widehat{\Delta}$ and $\Lambda_0 \in \widetilde{\mathfrak{h}}^*$ such that $\delta(D)=\Lambda_0(K)=1$ and $\delta(h)=\Lambda_0(h)=\delta(K)=\Lambda_0(D)=0$ for all $h \in \mathfrak{h}$. We have $\widetilde{\mathfrak{h}}^* = \mathfrak{h}^* \oplus \C\delta \oplus \C \Lambda_0$, $\widetilde{\nu}(\delta)=K$ and $\widetilde{\nu}(\Lambda_0)=D$. Denote by $\alpha^\vee = 2\widetilde{\nu}(\alpha)/(\alpha|\alpha) \in \widetilde{\mathfrak{h}}$ for $\alpha \in \widetilde{\mathfrak{h}}^*$ if $(\alpha|\alpha) \neq 0$. Then
\begin{align*}
&\alpha_i^\vee = 2\,\widetilde{\nu}(\alpha_i),\quad
i=1, \ldots, n-1,\\
&\alpha_n^\vee = 4\,\widetilde{\nu}(\alpha_n),\quad
\alpha_0^\vee = (\delta-\theta)^\vee = K - \theta^\vee.
\end{align*}
In particular, $\theta^\vee = \widetilde{\nu}(2\alpha_1 + \cdots + 2\alpha_n) = \alpha_1^\vee + \cdots + \alpha_{n-1}^\vee+\frac{1}{2}\alpha_n^\vee$. Recall that the Weyl vector $\rho \in \mathfrak{h}^*$ of $\mathfrak{osp}_{1|2n}$ is defined by $\rho(\alpha^\vee_i) = 1$ for all $\alpha_i \in \Pi$. Define an affine Weyl vector $\widehat{\rho} \in \widetilde{\mathfrak{h}}^*$ by $\widehat{\rho} = \rho + h^\vee\Lambda_0$, where
\begin{align*}
h^\vee = \rho(\theta^\vee) +1 = n + \frac{1}{2}
\end{align*}
is the dual Coxter number of $\mathfrak{osp}_{1|2n}$. Then $\widehat{\rho}(\alpha^\vee) = 1$ for all $\alpha \in \widehat{\Pi}$.

A root $\alpha \in \widehat{\Delta}$ is called even (resp. odd) if the root space $\widetilde{\mathfrak{g}}_\alpha$ is even (resp. odd). Let $\widehat{\Delta}_{\bar{0}}$ (resp. $\widehat{\Delta}_{\bar{1}}$) be the set of all even (resp. odd) roots. A root $\alpha \in \widehat{\Delta}_{\bar{0}}$ is called principal if $\alpha \in \widehat{\Pi}$ or $\alpha/2 \in \widehat{\Pi}$. Let $\widehat{\Pi}_\mathrm{pr}$ be the set of all principal roots. Then $\widetilde{\mathfrak{g}}^*_{\pm\alpha}$ are one-dimensional and they generate a Lie subalgebra $\mathfrak{sl}_2$ of $\widetilde{\mathfrak{g}}^*$. We have $\widehat{\Pi}_\mathrm{pr} = \{\alpha_0, \alpha_1, \ldots, \alpha_{n-1}, 2\alpha_n\}$. For $\alpha \in \widehat{\Pi}_\mathrm{pr}$, a simple reflection $r_\alpha \in \operatorname{GL}(\widetilde{\mathfrak{h}}^*)$ is well-defined. Then the affine Weyl group of $\widetilde{\mathfrak{g}}^*$ is defined by $\widehat{W}=\langle r_\alpha | \alpha \in \widehat{\Pi}_\mathrm{pr}\rangle_\mathrm{grp} \subset \operatorname{GL}(\widetilde{\mathfrak{h}}^*)$. For $\alpha, \lambda \in \widetilde{\mathfrak{h}}^*$, define a translation operator
\begin{align*}
t_\alpha(\lambda):=\lambda+\lambda(K)\alpha-((\lambda|\alpha)+\frac{1}{2}(\alpha|\alpha)\lambda(K))\delta.
\end{align*}
Then $r_{\alpha_0}r_{\theta}=t_{\theta}$ and $t_{w(\alpha)} = wt_{\alpha}w^{-1}$, $t_\alpha t_\beta = t_{\alpha+\beta}$ for $w \in W := \langle r_\alpha | \alpha \in \widehat{\Pi}_\mathrm{pr}\cap\Delta\rangle_\mathrm{grp}$ and $\alpha, \beta \in \widetilde{\mathfrak{h}}^*$. Thus $\widehat{W} \simeq W \ltimes M$ with $M=\Span_\Z\{w(\theta) \mid w \in W\} = \Span_\Z\{\alpha \in \Delta \mid (\alpha|\alpha)=2\}$. Let $\Delta^+$ be the set of positive roots in $\Delta$ and $\widehat{\Delta}^+ := \Delta^+ \sqcup\{\alpha + m\delta \mid \alpha \in \Delta, m \in \Z_{\geq1}\} \sqcup \{ m\delta \mid m \in \Z_{\geq0}\} \subset \widehat{\Delta}$. For $\lambda \in \widetilde{\mathfrak{h}}^*$, let
\begin{align*}
&\widehat{\Delta}(\lambda) := \{ \alpha \in \widehat{W}\widehat{\Pi}_\mathrm{pr} \mid \lambda(\alpha^\vee) \in \Z\},\quad
\widehat{\Delta}(\lambda)^+ := \widehat{\Delta}(\lambda) \cap \widehat{\Delta}^+,\\
&\widehat{\Pi}(\lambda) := \{ \beta \in \widehat{\Delta}(\lambda)^+ \mid {}^\nexists \alpha \in \widehat{\Delta}(\lambda)\ \mathrm{s.t.}\ 0<r_\alpha(\beta)<\beta\},
\end{align*}
where $\lambda \geq \mu \iff \lambda - \mu \in \Z_{\geq0}(\Delta\cap\Q_{\geq0}\widehat{\Pi}_\mathrm{pr})$. Then $\lambda$ is called admissible in the sense of \cite{KW08, GS} if
\begin{enumerate}
\item $(\lambda+\widehat{\rho})(\alpha^\vee) > 0$ for $\alpha \in \widehat{\Pi}(\lambda)$.
\item $\Q\widehat{\Delta}(\lambda) = \Q\widehat{\Delta}$.
\end{enumerate}
\begin{proposition}\label{prop:osp-adm}
For $k \in \C$, $\lambda = k \Lambda_0$ is admissible if and only if there exist coprime integers $a \in \Z$ and $b \in \Z_{\geq1}$ such that $\displaystyle k= \frac{a}{b}$ and
\begin{enumerate}
\item $b$ is odd $\Rightarrow$ $\displaystyle k + h^\vee \geq \frac{n+\frac{1}{2}}{b}$,
\item $b$ is even $\Rightarrow$ $\displaystyle k + h^\vee \geq \frac{2n-1}{b}$.
\end{enumerate}
\begin{proof}
The condition (2) implies that $k=\lambda(\alpha_0^\vee) \in \Q$. Hence $k$ is of the form $\displaystyle k= \frac{a}{b}$ with $a \in \Z$ and $b \in \Z_{\geq1}$ such that $(a, b) =1$. Notice that $(\alpha|\alpha)=1$ or $2$ if $\alpha \in \widehat{W}\widehat{\Pi}_\mathrm{pr}$. Thus $\widehat{W}\widehat{\Pi}_\mathrm{pr}\cap\Delta = \Delta_\mathrm{long}\sqcup\Delta_\mathrm{mid}$, where $\Delta_\mathrm{long} = \{ \alpha \in \Delta \mid (\alpha|\alpha)=2\}$ and $\Delta_\mathrm{mid} = \{ \alpha \in \Delta \mid (\alpha|\alpha)=1\}$. Since
\begin{align*}
\lambda((\alpha+m\delta)^\vee)=\lambda\left(\alpha^\vee+\frac{2mK}{(\alpha|\alpha)}\right)=\frac{2m}{(\alpha|\alpha)}\cdot\frac{a}{b}
\end{align*}
for $\alpha \in \Delta_\mathrm{long}\sqcup\Delta_\mathrm{mid}$, we have\\
\begin{align*}
\widehat{\Delta}(\lambda)=
\begin{cases}
\displaystyle \{bm\delta+\alpha|m\in\Z, \alpha \in \Delta_\mathrm{long}\sqcup\Delta_\mathrm{mid}\} &\ \mathrm{if}\ b\ \mathrm{is}\ \mathrm{odd},\\
\displaystyle \{bm\delta+\alpha\mid m\in\Z, \alpha \in \Delta_\mathrm{long}\}
\sqcup\left\{\frac{b}{2}m\delta+\alpha\,\middle|\, m\in\Z, \alpha \in \Delta_\mathrm{mid}\right\}&\ \mathrm{if}\ b\ \mathrm{is}\ \mathrm{even}.
\end{cases}
\end{align*}
Therefore $\lambda$ is admissible if and only if\\
(i) $(\lambda+\widehat{\rho})(\alpha^\vee) \in \Z_{\geq1}$ for all $\alpha \in \widehat{\Delta}(\lambda)^+$ if $b$ is odd;\\
(ii) $(\lambda+\widehat{\rho})(\alpha^\vee) \in \frac{1}{2}\Z_{\geq1}$ for all $\alpha \in \widehat{\Delta}(\lambda)^+$ if $b$ is even.

Note that we have the following useful formula
\begin{align*}
\widehat{\Delta}(\lambda)=\left\{\alpha \in \widehat{W}\widehat{\Pi}_\mathrm{pr} \,\middle|\, \alpha/2\notin\widehat{\Delta}\,\&\,(\lambda+\widehat{\rho})(\alpha^\vee)\in\Z\ \mathrm{or}\ \alpha/2\in\widehat{\Delta}\,\&\,(\lambda+\widehat{\rho})(\alpha^\vee) \in \Z + \frac{1}{2}\right\}.
\end{align*}
In case $b$ is odd, $\widehat{\Pi}(\lambda)=\{b\delta-\theta, \alpha_1, \ldots, \alpha_{n-1},2\alpha_n\}$. Then, since
\begin{align*}
(b\delta-\theta)^\vee = bK-\theta^\vee,\quad
\lambda+\widehat{\rho} = (k+h^\vee)\Lambda_0+\rho,\quad
\rho(\theta^\vee)=h^\vee-1,
\end{align*}
the condition (1) is equivalent to
\begin{align*}
(\lambda+\widehat{\rho})\left((b\delta-\theta)^\vee\right)=(\lambda+\widehat{\rho})(bK-\theta^\vee)=b(k+h^\vee)-(h^\vee-1)\geq1 \iff k + h^\vee \geq \frac{h^\vee}{b}.
\end{align*}

In case $b$ is even,
\begin{align*}
\widehat{\Pi}(\lambda)=
\begin{cases}
\{b\delta-\theta, 2\alpha_1\}&\ \mathrm{if}\ n =1,\\
\{\frac{b}{2}\delta-\theta_s, \alpha_1, \ldots, \alpha_{n-1},2\alpha_n\} &\ \mathrm{if}\ n \geq 2,
\end{cases}
\end{align*}
where $\theta_s = \theta-\alpha_1 = \alpha_1 + 2(\alpha_2 + \cdots + \alpha_n)$ is the highest root in $\Delta_\mathrm{mid}$. Then, since
\begin{align*}
&\left(\frac{b}{2}\delta-\theta_s\right)^\vee = 2\,\widetilde{\nu}\left(\frac{b}{2}\delta-\theta_s\right)=bK-\theta_s^\vee,\\
&\theta_s^\vee = 2\,\widetilde{\nu}\left(\alpha_1 + 2(\alpha_2 + \cdots + \alpha_n)\right) = \alpha_1^\vee+2(\alpha_2^\vee + \cdots + \alpha_{n-1}^\vee)+\alpha_n^\vee,\\
&\rho(\theta_s^\vee) = 1 + 2(n-2) + 1 = 2(n-1),
\end{align*}
the condition (1) is equivalent to the following:\\
(i) if $n=1$,
\begin{align*}
&(\lambda+\widehat{\rho})\left((b\delta-\theta)^\vee\right)=(\lambda+\widehat{\rho})(bK-\theta^\vee)=b(k+h^\vee)-\frac{1}{2}\geq\frac{1}{2}\\
&\iff k+h^\vee \geq \frac{1}{b};
\end{align*}
(ii) if $n \geq 2$,
\begin{align*}
&(\lambda+\widehat{\rho})\left(\left(\frac{b}{2}\delta-\theta_s\right)^\vee\right)=(\lambda+\widehat{\rho})(bK-\theta_s^\vee)=b(k+h^\vee)-2(n-1)\geq1 \\
&\iff k+h^\vee \geq \frac{2n-1}{b}.
\end{align*}
\end{proof}
\end{proposition}

\begin{remark}
Proposition \ref{prop:osp-adm} for $n=1$ recovers the statement of \cite{KW88} for admissible levels of $\mathfrak{osp}(1|2)$.
\begin{proof}
Recall that $k$ is an admissible level if $\lambda=k\Lambda_0$ is admissible. \cite[Example 2]{KW88} says that $\displaystyle k=\frac{m}{2}=\frac{t}{2u}$ with coprime integers $t \in \Z$ and $u \in \Z_{\geq1}$ is an admissible level if and only if $3u+t-1\geq0$ if $t$ is odd and $3u+t-3\geq0$ if $t$ is even. Let $a \in \Z$ and $b \in \Z_{\geq1}$ be coprime integers such that $\displaystyle k=\frac{t}{2u}=\frac{a}{b}$. If $t$ is odd, since $(t, 2u)=1$, $a=t$ and $b=2u$. Hence $b$ is even. Then $3u+t-1\geq0$ is equivalent to
\begin{align*}
3b+2a-3\geq0 \iff \frac{a}{b}+\frac{3}{2}\geq\frac{3}{2b}.
\end{align*}
If $t$ is even, $u$ is odd, and since $(t/2, u)=1$, $a=t/2$ and $b=u$. Hence $b$ is odd. Then $3u+t-3\geq0$ is equivalent to
\begin{align*}
\frac{3}{2}b+a-1\geq0 \iff \frac{a}{b}+\frac{3}{2}\geq\frac{1}{b}.
\end{align*}
Therefore the statement in \cite{KW88} is the same as Proposition \ref{prop:osp-adm} for $n=1$.
\end{proof}
\end{remark}

\section{Formal Characters}

\subsection{Characters of Modules in $\mathcal{O}$}
Let $\widetilde{\mathfrak{g}}$ be a Kac-Moody Lie superalgebra in the sense of Serganova \cite{Serganova} and $\widetilde{\mathfrak{h}}$ be the Cartan subalgebra of $\widetilde{\mathfrak{g}}$. Then we can define the category $\mathcal{O}$ as certain full subcategory of $\widetilde{\mathfrak{g}}$-modules in the same way as in \cite{DGK}, and the characters
\begin{align*}
\operatorname{ch} M = \sum_{\lambda \in \widetilde{\mathfrak{h}}^*} (\dim M_\lambda) \mathrm{e}^\lambda,\quad
M \in \mathcal{O},
\end{align*}
where $M_\lambda$ is the weight space of $M$ for $\widetilde{\mathfrak{h}}$ of the weight $\lambda$. See e.g. \cite[Section 1]{GS} for the details.

\subsection{Characters of $V^k(\mathfrak{g})$-modules}\label{sec:affine-VA-mod}
Let $\mathfrak{g}$ be a finite-dimensional simple Lie superalgebra with the normalized even supersymmetric invariant bilinear form $( \cdot | \cdot )$ such that $( \theta | \theta ) = 2$ for the highest root $\theta$ of $\mathfrak{g}$, and $\widehat{\mathfrak{g}} = \mathfrak{g}[t, t^{-1}] \oplus \C K$ be the affine Lie superalgebra of $\mathfrak{g}$ with the central element $K$. For any finite-dimensional highest-weight $\mathfrak{g}$-module $E$ and $k \in \C$, we define the induced $\widehat{\mathfrak{g}}$-module $V^k_\mathfrak{g}(E)$ by
\begin{align*}
V^k_\mathfrak{g}(E) = U( \widehat{\mathfrak{g}} ) \underset{ U( \mathfrak{g}[t] \oplus \C K ) }{\otimes} E,
\end{align*}
where we consider $E$ as a $\mathfrak{g}[t] \oplus \C K$-module by $\mathfrak{g}[t]t = 0$ and $K = k$, and $U( \mathfrak{a} )$ denotes the universal enveloping algebra of $\mathfrak{a}$ for any Lie superalgebra $\mathfrak{a}$. $V^k_\mathfrak{g}(E)$ is called the local Weyl $\widehat{\mathfrak{g}}$-module induced from $E$ at level $k$. Set
\begin{align*}
V^k(\mathfrak{g}) = V^k_\mathfrak{g}(\C)
\end{align*}
for the trivial $\mathfrak{g}$-module $\C$. It is well-known that $V^k(\mathfrak{g})$ has a vertex superalgebra structure, and then $V^k_\mathfrak{g}(E)$ is a $V^k(\mathfrak{g})$-module. Let $h^\vee$ be the dual Coxeter number of $\mathfrak{g}$. If $k+h^\vee \neq 0$, the Sugawara construction defines the Virasoro field $L(z) = \sum_{m \in \Z} L_m z^{-m-2}$ on $V^k(\mathfrak{g})$ so that any $V^k(\mathfrak{g})$-module has a Virasoro module structure with the central charge $\frac{ k \operatorname{sdim}( \mathfrak{g} )}{k+h^\vee}$. From the point of view of vertex superalgebras, the formal character of $V^k_\mathfrak{g}(E)$ is defined as follows:
\begin{align*}
\operatorname{\widetilde{ch}} V^k_\mathfrak{g}(E) = q^{m_k^\mathfrak{g}(\mu)} \operatorname{ch} V^k_\mathfrak{g}(E),\quad
m_k^\mathfrak{g}(\mu) = \frac{ (\mu \mid \mu + 2 \rho) }{ 2 (k+h^\vee) },
\end{align*}
where $\mu$ is the highest weight of $E$, and $\rho$ is the Weyl vector of $\mathfrak{g}$. By definition, the specialization of $\operatorname{\widetilde{ch}} V^k_\mathfrak{g}(E)$ by $\mathrm{e}^\alpha \mapsto \mathrm{e}^{(\alpha \mid \lambda)}$ coincides with the trace function of $q^{L_0}\mathrm{e}^\lambda$ for $\lambda \in \mathfrak{h}^*$ on $V^k_\mathfrak{g}(E)$.

\subsection{Characters of $\mathcal{W}^k (\mathfrak{g})$-modules}\label{sec:W-alg-char}
Suppose that $\mathfrak{g}$ is a simple Lie algebra. Let $\mathfrak{n}_+$ (resp. $\mathfrak{n}_-$) be the sum of positive (resp. negative) root vector spaces of $\mathfrak{g}$, $f$ a principal nilpotent element of $\mathfrak{g}$ in $\mathfrak{n}_-$, and $H_{DS}^\bullet (?)$ the Drinfeld-Sokolov reduction cohomology functor associated to $f$, which is defined by
\begin{align*}
H_{DS}^\bullet (M) = H^{\frac{\infty}{2}+\bullet}(L\mathfrak{n}_+, M \otimes \C_\chi),\quad
M \in \mathcal{O},
\end{align*}
where $L\mathfrak{n}_+ = \mathfrak{n}_+[t, t^{-1}]$, $H^{\frac{\infty}{2}+\bullet}(L\mathfrak{n}_+, ?)$ is the semi-infinite cohomology of $L\mathfrak{n}_+$-modules, and $\C_\chi$ is the one-dimensional $L\mathfrak{n}_+$-module by $x t^m \mapsto \delta_{m, -1}(f|x)$. Then
\begin{align*}
\mathcal{W}^k (\mathfrak{g}) = H_{DS}^\bullet ( V^k(\mathfrak{g}) )
\end{align*}
has a vertex algebra structure, called the (principal) $\mathcal{W}$-algebra associated to $\mathfrak{g}$ at level $k$. We have $H_{DS}^i ( V^k(\mathfrak{g}) ) = 0$ unless $i = 0$. If $k+h^\vee \neq 0$, $\mathcal{W}^k (\mathfrak{g})$ also has a Virasoro field of the central charge $\operatorname{rank}(\mathfrak{g}) - 12\frac{| \rho - (k+h^\vee) \rho^\vee |^2}{k+h^\vee}$, where $\rho^\vee$ is the Weyl covector of $\mathfrak{g}$. For any $\mathcal{W}^k(\mathfrak{g})$-module $W$ whose $L_0$-eigenspaces are finite-dimensional, we define the formal character of $W$ by
\begin{align*}
\operatorname{\widetilde{ch}} W = \operatorname{tr}_W ( q^{L_0} ).
\end{align*}
Consider a $\mathcal{W}^k (\mathfrak{g})$-module $H_{DS}^\bullet ( M )$ for a highest-weight $V^k(\mathfrak{g})$-module $M$ with the highest weight $\mu$. Using the Euler-Poincar\'e principle, we have
\begin{align*}
&\operatorname{\widetilde{ch}} H_{DS}^\bullet ( M ) = \frac{q^{m_k^\mathfrak{g}(\mu)}}{\prod_{j=1}^\infty (1 - q^j)^{\operatorname{rank}(\mathfrak{g})}} \left.(\widehat{R}\operatorname{ch} M)\right|_{\mathrm{e}^\alpha \mapsto q^{-(\alpha \mid \rho^\vee)}},
\end{align*}
where $\widehat{R}$ is the Weyl denominator of $\widehat{\mathfrak{g}}$.

\section{Characters of modules of $\widehat{\mathfrak{osp}}_{1|2n}$ and $\widehat{\mathfrak{sp}}_{2n}$}

\subsection{Settings for $\widehat{\mathfrak{osp}}_{1|2n}$ and $\widehat{\mathfrak{sp}}_{2n}$}
Consider the cases that $\widehat{\mathfrak{g}} = \widehat{\mathfrak{osp}}_{1|2n} $ and $\widehat{\mathfrak{g}} = \widehat{\mathfrak{sp}}_{2n}$. Since $\mathfrak{sp}_{2n} \subset \mathfrak{osp}_{1|2n}$, we may identify the Cartan subalgebra of $\mathfrak{sp}_{2n}$ with that of $\mathfrak{osp}_{1|2n}$, which we denote by $\mathfrak{h}$. Let $\Pi^\mathfrak{osp} = \{\alpha_1,\ldots,\alpha_n\}$ be a set of simple roots of $\mathfrak{osp}_{1|2n}$ such that $\Pi^\mathfrak{sp} = \{\alpha_1,\ldots,\alpha_{n-1},2\alpha_n\}$ forms a set of simple roots of $\mathfrak{sp}_{2n}$. The Dynkin diagrams corresponding to $\Pi^\mathfrak{osp}, \Pi^\mathfrak{sp}$ are the following:
\begin{align*}
\quad\\
\setlength{\unitlength}{1mm}
\begin{picture}(0,0)(20,10)
\put(0,10){\circle{2}}
\put(-1,5){\footnotesize$\alpha_1$}
\put(1,10.3){\line(1,0){8}}
\put(10,10){\circle{2}}
\put(9,5){\footnotesize$\alpha_2$}
\put(11,10.3){\line(1,0){6}}
\put(18.5,9.4){$\cdot$}
\put(20,9.4){$\cdot$}
\put(21.5,9.4){$\cdot$}
\put(24.5,10.3){\line(1,0){6}}
\put(32,10){\circle{2}}
\put(31,5){\footnotesize$\alpha_{n-1}$}
\put(33,10.5){\line(1,0){8}}
\put(33,9.5){\line(1,0){8}}
\put(35.5,9){\Large$>$}
\put(42,10){\circle*{2}}
\put(41,5){\footnotesize$\alpha_{n}$}
\put(44.5,9){}
\end{picture}
\quad \quad\quad\quad\quad\quad\quad\quad\quad\quad\quad\quad       \text{and}    \quad\quad\quad\quad\quad\quad\quad\quad\quad \quad\quad 
\setlength{\unitlength}{1mm}
\begin{picture}(0,0)(20,10)
\put(0,10){\circle{2}}
\put(-1,5){\footnotesize$\alpha_1$}
\put(1,10.3){\line(1,0){8}}
\put(10,10){\circle{2}}
\put(9,5){\footnotesize$\alpha_2$}
\put(11,10.3){\line(1,0){6}}
\put(18.5,9.4){$\cdot$}
\put(20,9.4){$\cdot$}
\put(21.5,9.4){$\cdot$}
\put(24.5,10.3){\line(1,0){6}}
\put(32,10){\circle{2}}
\put(31,5){\footnotesize$\alpha_{n-1}$}
\put(33,10.5){\line(1,0){8}}
\put(33,9.5){\line(1,0){8}}
\put(35.5,9){\Large$<$}
\put(42,10){\circle{2}}
\put(41,5){\footnotesize$2\alpha_{n}$}
\put(44.5,9){.}
\end{picture}\\
\end{align*}
The normalized invariant form $( \cdot | \cdot )$ on $\mathfrak{osp}_{1|2n}$ satisfies that $( 2\alpha_{n} | 2\alpha_{n} ) = 2$. Then the invariant form restricted to $\mathfrak{sp}_{2n}$ is the same as the standard normalized form on $\mathfrak{sp}_{2n}$. Let $Q^\mathfrak{osp}$ (resp. $Q^\mathfrak{sp}$) be the root lattice of $\mathfrak{osp}_{1|2n}$ (resp. $\mathfrak{sp}_{2n}$), and $P^\mathfrak{osp}$ (resp. $P^\mathfrak{sp}$) be the weight lattice of $\mathfrak{osp}_{1|2n}$ (resp. $\mathfrak{sp}_{2n}$). It is easy to see that $Q^\mathfrak{osp} = P^\mathfrak{osp} = P^\mathfrak{sp}$, which we denote by $P$. Let $\Delta_+$, $\Delta_{+, \bar{0}}$, $\Delta_{+, \bar{1}}$ be the set of positive roots, even positive roots, odd positive roots of $\mathfrak{osp}_{1|2n}$ respectively. Set
\begin{align*}
\rho_\mathfrak{osp} = \rho_\mathfrak{sp} -\rho_{\bar{1}},\quad
\rho_\mathfrak{sp} = \frac{1}{2}\sum_{\alpha \in \Delta_{+, \bar{0}}}\alpha,\quad
\rho_{\bar{1}} = \frac{1}{2}\sum_{\alpha \in \Delta_{+, \bar{1}}}\alpha.
\end{align*}
Let $\rho^\vee_\mathfrak{sp}$ be the element in $\mathfrak{h}^*$ corresponding to the Weyl covector of $\mathfrak{sp}_{2n}$, that is, $(\rho^\vee_\mathfrak{sp}|\alpha) = 1$ for all $\alpha \in \Pi^\mathfrak{sp}$. Then we have
\begin{align*}
\rho_\mathfrak{osp} = \frac{1}{2} \rho^\vee_\mathfrak{sp}
\end{align*}
by direct computations.

\subsection{Verma modules}
For $k\in\C$ and $\lambda \in \mathfrak{h}^*$, let $M^k_\mathfrak{osp}(\lambda)$ (resp. $M^k_\mathfrak{sp}(\lambda)$) be the Verma module of $\widehat{\mathfrak{osp}}_{1|2n}$ (resp. $\widehat{\mathfrak{sp}}_{2n}$) with the highest weight $\lambda$ at level $k$. Then
\begin{align*}
&\operatorname{ch} M^k_\mathfrak{osp}(\lambda) = \frac{\mathrm{e}^\lambda}{\widehat{R}_\mathfrak{osp}},\quad
\operatorname{ch} M^k_\mathfrak{sp}(\lambda) = \frac{\mathrm{e}^\lambda}{\widehat{R}_\mathfrak{sp}},\\
&\widehat{R}_\mathfrak{osp}^{-1} = \widehat{R}_\mathfrak{sp}^{-1} \prod_{\alpha \in \Delta_{+, \bar{1}}}\prod_{j=1}^\infty (1 + \mathrm{e}^\alpha q^j) (1 + \mathrm{e}^{-\alpha} q^{j-1}),\\
&\widehat{R}_\mathfrak{sp} = \prod_{\alpha \in \Delta_{+, \bar{0}}}\prod_{j=1}^\infty (1 - q^j)^n (1 - \mathrm{e}^\alpha q^j) (1 - \mathrm{e}^{-\alpha} q^{j-1}),
\end{align*}
where $q = \mathrm{e}^{-\delta}$ as usual.
\begin{lemma}
\begin{align*}
\widehat{R}_\mathfrak{osp}^{-1} = \frac{\widehat{R}_\mathfrak{sp}^{-1}}{\prod_{j=1}^\infty (1 - q^j)^n} \sum_{\lambda \in P} \mathrm{e}^{\lambda} q^{(\lambda \mid \lambda + 2 \rho_{\bar{1}})}.
\end{align*}
\begin{proof}
Let $\lambda_i = \sum_{j = i}^n \alpha_j$. Then $\Delta_{+, \bar{1}} = \{\lambda_1,\ldots,\lambda_n\}$. We have $(\lambda_i | \lambda_j) = \frac{1}{2} \delta_{i, j}$, $2 \rho_{\bar{1}} = \sum_{i=1}^n \lambda_i$ and $P = \bigoplus_{i=1}^n \Z \lambda_i$. Using the Jacobi triple product identity
\begin{align*}
\prod_{j=1}^\infty (1 - q^j)^n (1 + u q^j) (1 + u^{-1} q^{j-1}) = \sum_{m \in \Z} u^m q^{\frac{1}{2} m(m+1)}
\end{align*}
for $u = \mathrm{e}^{\lambda_i}$, it follows that
\begin{multline*}
\widehat{R}_\mathfrak{sp} \prod_{j=1}^\infty (1 - q^j)^n \widehat{R}_\mathfrak{osp}^{-1}
= \prod_{i=1}^n \sum_{m_i \in \Z}\mathrm{e}^{m_i \lambda_i} q^{\frac{1}{2} m_i(m_i+1)}
= \sum_{m_1,\ldots,m_n \in \Z}\mathrm{e}^{\sum_{i=1}^n m_i \lambda_i} q^{\frac{1}{2} \sum_{i=1}^n m_i(m_i+1)}
= \sum_{\lambda \in P} \mathrm{e}^{\lambda} q^{(\lambda \mid \lambda + 2 \rho_{\bar{1}})}.
\end{multline*}
\end{proof}
\end{lemma}
\begin{corollary}\label{cor:Verma-osp-sp} For any $\mu \in P$,
\begin{align*}
\operatorname{ch} M^k_\mathfrak{osp}(\mu) = \sum_{\lambda \in P} \operatorname{ch} M^k_\mathfrak{sp}(\lambda) B^\lambda_\mu,\quad
B^\lambda_\mu = \frac{q^{(\lambda-\mu \mid \lambda-\mu+2 \rho_{\bar{1}})}}{\prod_{j=1}^\infty (1 - q^j)^n}.
\end{align*}
\end{corollary}
Let $W$ be the Weyl group of $\mathfrak{sp}_{2n}$ and $|\lambda|^2 = (\lambda | \lambda)$ for $\lambda \in \mathfrak{h}^*$. Since $|\rho_{\bar{1}}|^2 = \frac{n}{8}$, we have
\begin{align*}
B^\lambda_\mu = \frac{q^{|\lambda-\mu+\rho_{\bar{1}}|^2-\frac{n}{8}}}{\prod_{j=1}^\infty (1 - q^j)^n}.
\end{align*}
Thus,
\begin{align*}
B^\lambda_\mu = B^{\lambda+\nu}_{\mu+\nu},\quad
\nu \in \mathfrak{h}^*;\quad
B^\lambda_{w(\mu)} = B^{w^{-1}(\lambda + \rho_{\bar{1}}) - \rho_{\bar{1}}}_\mu,\quad
w \in W.
\end{align*}

\subsection{Local Weyl modules}
Let $P_+$ be the set of dominant weights in $P$. For $k \in \C$ and $\mu \in P_+$, let $E_\mu^\mathfrak{osp}$ (resp. $E_\mu^\mathfrak{sp}$) be the finite-dimensional simple $\mathfrak{osp}_{1|2n}$-module (resp. $\mathfrak{sp}_{2n}$-module) with the highest weight $\mu$, and $V^k_\mathfrak{osp}(\mu)$ (resp. $V^k_\mathfrak{sp}(\mu)$) be the local Weyl module of $\widehat{\mathfrak{osp}}_{1|2n}$ (resp. $\widehat{\mathfrak{sp}}_{2n}$) induced from $E_\mu$ at level $k$. See Section \ref{sec:affine-VA-mod} for the definitions of local Weyl modules. We have
\begin{align*}
V^k_\mathfrak{osp}(\mu) \simeq U( \mathfrak{osp}_{1|2n}[t^{-1}]t^{-1} ) \otimes E_\mu^\mathfrak{osp},\quad
V^k_\mathfrak{sp}(\mu) \simeq U( \mathfrak{sp}_{2n}[t^{-1}]t^{-1} ) \otimes E_\mu^\mathfrak{sp}
\end{align*}
as vector spaces. Thus,
\begin{align*}
&\operatorname{ch} V^k_\mathfrak{osp}(\mu) = \sum_{w \in W} (-1)^{l(w)} \operatorname{ch} M^k_\mathfrak{osp}( w(\mu+\rho_\mathfrak{osp}) - \rho_\mathfrak{osp} ),\\
&\operatorname{ch} V^k_\mathfrak{sp}(\mu) = \sum_{w \in W} (-1)^{l(w)} \operatorname{ch} M^k_\mathfrak{sp}( w(\mu+\rho_\mathfrak{sp}) - \rho_\mathfrak{sp} ),
\end{align*}
where $l(w)$ is the length of $w$ in $W$. Using the equalities
\begin{align*}
B^\lambda_{ w( \mu + \rho_\mathfrak{osp} ) - \rho_\mathfrak{osp} }
= B^{ \lambda + \rho_\mathfrak{osp} }_{ w( \mu + \rho_\mathfrak{osp} ) }
= B^{ w^{-1}(\lambda + \rho_\mathfrak{sp}) - \rho_{\bar{1}} }_{ \mu + \rho_\mathfrak{osp} }
= B^{ w^{-1}(\lambda + \rho_\mathfrak{sp}) - \rho_\mathfrak{sp} }_\mu,
\end{align*}
and Corollary \ref{cor:Verma-osp-sp}, it follows that
\begin{align}\label{eq:ch-osp-sp}
\operatorname{ch} V^k_\mathfrak{osp}(\mu) 
= \sum_{\lambda \in P} \operatorname{ch} M^k_\mathfrak{sp}(\lambda) \sum_{w \in W} (-1)^{l(w)} B^{ w(\lambda + \rho_\mathfrak{sp}) - \rho_\mathfrak{sp} }_\mu.
\end{align}
Define a real form $\mathfrak{h}^*_\R$ of $\mathfrak{h}^*$ and a Weyl chamber $C$ in $\mathfrak{h}^*_\R$ by
\begin{align*}
\mathfrak{h}^*_\R = P \otimes_\Z \R,\quad
C = \{ \lambda \in \mathfrak{h}^*_\R \mid (\lambda | \alpha^\vee ) >0\ \mathrm{for}\ \mathrm{all}\ \alpha \in \Delta_+\},
\end{align*}
where $\alpha^\vee = 2 \alpha/(\alpha|\alpha) \in \mathfrak{h}^*$. Then
\begin{align*}
W(C) = \bigcup_{w \in W} w(C) = \{ \lambda \in \mathfrak{h}^*_\R \mid (\lambda | \alpha ) \neq 0\ \mathrm{for}\ \mathrm{all}\ \alpha \in \Delta\},\quad
W(\overline{C}) = \bigcup_{w \in W} w(\overline{C}) = \mathfrak{h}^*_\R.
\end{align*}
An element $\lambda \in \mathfrak{h}^*_\R$ is called regular if $\lambda \in W(C)$, and is called singular unless $\lambda$ is regular.
\begin{lemma}\label{lemma:P-reg-sing}
Let $P^\mathrm{reg}$ be the set of regular elements in $P$, $P^\mathrm{sing}$ the set of singular elements in $P$, $P_+^\mathrm{reg} = P_+ \cap P^\mathrm{reg}$, and $P_+^\mathrm{sing} = P_+ \cap P^\mathrm{sing}$. Then
\begin{enumerate}
\item $P = P^\mathrm{reg} \sqcup P^\mathrm{sing}$ and $P_+ = P_+^\mathrm{reg} \sqcup P_+^\mathrm{sing}$.
\item $P_+^\mathrm{reg} = P_+ + \rho_\mathfrak{sp}$.
\item $P^\mathrm{reg} = W(P_+^\mathrm{reg}) \simeq W \times P_+^\mathrm{reg}$.
\item $P^\mathrm{sing} = W(P_+^\mathrm{sing})$.
\end{enumerate}
\begin{proof}
(1) is trivial. First, we consider (2). Since $P_+ = P \cap \overline{C}$, (2) follows from the fact that $P_+$ coincides with the set of integral dominant weights of $\mathfrak{sp}_{2n}$. Next, we consider (3) and (4). Since $\mathfrak{h}^*_\R = W(\overline{C})$ and $P$ is $W$-invariant, we have $P = W(P \cap \overline{C}) = W(P_+)$. Hence $P^\mathrm{reg} = W(P_+^\mathrm{reg})$ and $P^\mathrm{sing} = W(P_+^\mathrm{sing})$. This proves (4) and the first equality of (3). Finally, we show the last isomorphism of (3). Using the fact that $W$ acts transitively on the set of Weyl chambers of $\mathfrak{sp}_{2n}$, it follows that the map $W \times P_+^\mathrm{reg} \ni (w,\lambda) \mapsto w(\lambda) \in P^\mathrm{reg}$ gives an isomorphism. Therefore (3) follows.
\end{proof}
\end{lemma}
Set
\begin{align*}
w \circ \lambda = w(\lambda + \rho_\mathfrak{sp}) - \rho_\mathfrak{sp},\quad
w \in W,\quad
\lambda \in \mathfrak{h}^*.
\end{align*} 
\begin{proposition}\label{prop:ch-osp-sp} Suppose that $k \notin \Q$. For $\mu \in P_+$, we have
\begin{align*}
\operatorname{ch} V^k_\mathfrak{osp}(\mu) = \sum_{\lambda \in P_+} \operatorname{ch} V^k_\mathfrak{sp}(\lambda) \sum_{w \in W} (-1)^{l(w)} B^{ w \circ \lambda }_\mu.
\end{align*}
\begin{proof}
By \eqref{eq:ch-osp-sp} and Lemma \ref{lemma:P-reg-sing}, we have
\begin{equation}\label{eq:proof-of-ch-prop-1} 
\begin{split}
&\operatorname{ch} V^k_\mathfrak{osp}(\mu)
= \sum_{\lambda \in P_+} \sum_{w_1 \in W} \operatorname{ch} M^k_\mathfrak{sp}(w_1 \circ \lambda) \sum_{w_2 \in W} (-1)^{l(w_2)} B^{ w_2 \circ w_1 \circ \lambda}_\mu\\
&\quad\quad \quad + \sum_{\lambda \in P_+^\mathrm{sing} - \rho_\mathfrak{sp}} \sum_{w_1 \in S_\lambda} \operatorname{ch} M^k_\mathfrak{sp}(w_1 \circ \lambda) \sum_{w_2 \in W} (-1)^{l(w_2)} B^{ w_2 \circ w_1 \circ \lambda}_\mu\\
&= \sum_{\lambda \in P_+} \operatorname{ch} V^k_\mathfrak{sp}(\lambda) \sum_{w \in W} (-1)^{l(w)} B^{ w \circ \lambda}_\mu
+ \sum_{\lambda \in P_+^\mathrm{sing} - \rho_\mathfrak{sp}} \sum_{w_1 \in S_\lambda} (-1)^{l(w_1)} \operatorname{ch} M^k_\mathfrak{sp}(w_1 \circ \lambda) \sum_{w_2 \in W} (-1)^{l(w_2)} B^{ w_2 \circ \lambda}_\mu,
\end{split}
\end{equation}
where $S_\lambda$ is a set of representatives of $W / \Ker\left( W \ni w \mapsto w(\lambda + \rho_\mathfrak{sp}) \in W(\lambda + \rho_\mathfrak{sp}) \right)$ in $W$. Since $k \notin \Q$, the Kazhdan-Lusztig category is semisimple and every simple object is isomorphic to $V^k_\mathfrak{sp}(\lambda)$ for certain $\lambda \in P_+$ so that
\begin{align*}
V^k_\mathfrak{osp}(\mu) \simeq \bigoplus_{\lambda \in P_+} V^k_\mathfrak{sp}(\lambda) \otimes \Hom_{\widehat{\mathfrak{sp}}_{2n}}\left( V^k_\mathfrak{sp}(\lambda), V^k_\mathfrak{osp}(\mu) \right).
\end{align*}
Then each coefficient of $\mathrm{e}^{w \circ \lambda}$ for $w \in W$ and $\lambda \in P_+^\mathrm{sing} - \rho_\mathfrak{sp}$ in $\operatorname{ch} V^k_\mathfrak{osp}(\mu)$ must be zero. This implies that the second term in the last equation in \eqref{eq:proof-of-ch-prop-1} is equal to zero. Therefore we complete the proof.
\end{proof}
\end{proposition}
\begin{corollary} For $\lambda \in P_+^\mathrm{sing} - \rho_\mathfrak{sp}$ and $\mu \in P_+$,
\begin{align*}
\sum_{w \in W} (-1)^{l(w)} B^{ w \circ \lambda }_\mu = 0.
\end{align*}
\end{corollary}

\section{Main Results}

\subsection{Branching functions}
Let $h^\vee_\mathfrak{osp}$, $h^\vee_\mathfrak{sp}$ be the dual Coxeter numbers of $\mathfrak{osp}_{1|2n}$, $\mathfrak{sp}_{2n}$ respectively. Then $h^\vee_\mathfrak{osp} = n + \frac{1}{2}$, $h^\vee_\mathfrak{sp} = n + 1$. From now on, we assume that $k \notin \Q$. Especially, $k+h^\vee_\mathfrak{osp} \neq 0$ and $k+h^\vee_\mathfrak{sp} \neq 0$. Set
\begin{align*}
m_k^\mathfrak{osp}(\mu) = \frac{(\mu \mid \mu + 2 \rho_\mathfrak{osp})}{2 (k+h^\vee_\mathfrak{osp})},\quad
m_k^\mathfrak{sp}(\mu) = \frac{(\mu \mid \mu + 2 \rho_\mathfrak{sp})}{2 (k+h^\vee_\mathfrak{sp})},\quad
\mu \in P_+.
\end{align*}
Now, it follows from Proposition \ref{prop:ch-osp-sp} (see Corollary \ref{cor:Verma-osp-sp} for the definitions of $B^\lambda_\mu$) that
\begin{align}
\label{eq:branching-1} &\operatorname{\widetilde{ch}} V^k_\mathfrak{osp} (\mu) = \sum_{\lambda \in P_+} \operatorname{\widetilde{ch}} V^k_\mathfrak{sp} (\lambda) \sum_{w \in W} (-1)^{l(w)} \widetilde{B}^{ w \circ \lambda }_\mu,\quad
\widetilde{B}^{ w \circ \lambda }_\mu = \frac{ q^{ \Delta_{w \circ \lambda, \mu} } }{\prod_{j=1}^\infty (1 - q^j)^n},\\
\label{eq:branching-2} &\Delta_{w \circ \lambda, \mu} = m_k^\mathfrak{osp}(\mu) - m_k^\mathfrak{sp}(\lambda) +(w \circ \lambda - \mu | w \circ \lambda - \mu + 2 \rho_{\bar{1}}).
\end{align}
We call the coefficients of $\operatorname{\widetilde{ch}} V^k_\mathfrak{sp} (\lambda)$ in $\operatorname{\widetilde{ch}} V^k_\mathfrak{osp} (\mu)$ the branching functions. Define $\ell \in \C$ by
\begin{align}\label{eq:ell}
\frac{1}{ k+h^\vee_\mathfrak{sp} } + \frac{1}{ \ell+h^\vee_\mathfrak{sp} } = 2.
\end{align}
\begin{lemma}\label{lemma:delta-lambda-mu}
\begin{align*}
\Delta_{w \circ \lambda, \mu} = m_\ell^\mathfrak{sp}\left( w \circ \lambda - 2(\ell+h^\vee_\mathfrak{sp}) \mu \right) - \left( w \circ \lambda - 2(\ell+h^\vee_\mathfrak{sp}) \mu \mid \rho^\vee_\mathfrak{sp} \right).
\end{align*}
\begin{proof}
Direct computations.
\end{proof}
\end{lemma}

\subsection{Characters of simple $\mathcal{W}^\ell(\mathfrak{sp}_{2n})$-modules}
For $\nu \in \mathfrak{h}^*$, let $L_\ell^\mathfrak{sp}(\nu)$ be the simple quotient of $M^\ell_\mathfrak{sp}(\nu)$, and $\mathbb{W}^\ell_\mathfrak{sp}(\nu)$ be the Wakimoto module of $\widehat{\mathfrak{sp}}_{2n}$ with the highest weight $\nu$ at level $\ell$ \cite{Wakimoto, FF88, Frenkel05}. Consider a $\widehat{\mathfrak{sp}}_{2n}$-module $L_\ell^\mathfrak{sp}\left( \lambda - 2( \ell+h^\vee_\mathfrak{sp} )\mu \right)$ for $\lambda, \mu \in P_+$. By \cite{Frenkel91} (see also \cite[Proposition 4.2]{ACL}), we have a resolution of $L_\ell^\mathfrak{sp}(\lambda - 2( \ell+h^\vee_\mathfrak{sp} )\mu)$ of the form
\begin{align}
\label{eq:char-simple-1}&0 \rightarrow L_\ell^\mathfrak{sp}\left( \lambda - 2( \ell+h^\vee_\mathfrak{sp} )\mu \right) \rightarrow C_0 \rightarrow C_1 \rightarrow \cdots \rightarrow C_{n^2} \rightarrow 0,\\
\label{eq:char-simple-2}&C_i = \bigoplus_{\begin{subarray}{c} w \in W \\ l(w) = i \end{subarray}} \mathbb{W}^\ell_\mathfrak{sp} ( w \circ \lambda - 2( \ell + h^\vee_\mathfrak{sp} )\mu).
\end{align}
Thus,
\begin{align*}
\operatorname{ch} L_\ell^\mathfrak{sp}\left( \lambda - 2( \ell+h^\vee_\mathfrak{sp} )\mu \right) = \sum_{w \in W} (-1)^{l(w)} \operatorname{ch} \mathbb{W}^\ell_\mathfrak{sp} \left( w \circ \lambda - 2( \ell + h^\vee_\mathfrak{sp} )\mu \right).
\end{align*}
Recall the Drinfeld-Sokolov reduction cohomology functor $H_{DS}^\bullet (?)$ introduced in Section \ref{sec:W-alg-char}. It follows from \cite{Arakawa07} that $H_{DS}^0 ( L_\ell^\mathfrak{sp}\left( \lambda - 2( \ell+h^\vee_\mathfrak{sp} )\mu \right)  )$ is a simple $\mathcal{W}^\ell(\mathfrak{sp}_{2n})$-module, which we denote by $\mathbf{L}_\ell(\chi_{\lambda - 2( \ell+h^\vee_\mathfrak{sp} )\mu})$.
\begin{theorem}
\label{thm:character}
Suppose that $k \notin \Q$ and $\ell \in \C$ defined by \eqref{eq:ell}. For $\mu \in P_+$,
\begin{align*}
\operatorname{\widetilde{ch}} V^k_\mathfrak{osp} (\mu) = \sum_{\lambda \in P_+} \operatorname{\widetilde{ch}} V^k_\mathfrak{sp} (\lambda) \operatorname{\widetilde{ch}} \mathbf{L}_\ell(\chi_{\lambda - 2( \ell+h^\vee_\mathfrak{sp} )\mu}).
\end{align*}
\begin{proof}
First of all, we have $H_{DS}^i ( L_\ell^\mathfrak{sp}\left( \lambda - 2( \ell+h^\vee_\mathfrak{sp} )\mu \right)  ) = 0$ for $i \neq 0$ \cite{Arakawa04} and $H_{DS}^i ( \mathbb{W}^\ell_\mathfrak{sp} \left( w \circ \lambda - 2( \ell + h^\vee_\mathfrak{sp} )\mu \right) ) = 0$ for $i \neq 0$ \cite{FF92}. By applying the functor $H_{DS}^0(?)$ to \eqref{eq:char-simple-1}, we obtain the exact sequence
\begin{align*}
0 \rightarrow \mathbf{L}_\ell(\chi_{\lambda - 2( \ell+h^\vee_\mathfrak{sp} )\mu}) \rightarrow H_{DS}^0(C_0) \rightarrow H_{DS}^0(C_1) \rightarrow \cdots \rightarrow H_{DS}^0(C_{n^2}) \rightarrow 0,
\end{align*}
and, using the Euler-Poincar\'e principle, the character formulas
\begin{align*}
\operatorname{\widetilde{ch}} H_{DS}^0(C_i) = \sum_{\begin{subarray}{c} w \in W \\ l(w) = i \end{subarray}} \frac{q^{m_\ell^\mathfrak{sp}\left( w \circ \lambda - 2(\ell+h^\vee_\mathfrak{sp}) \mu \right)}}{\prod_{j=1}^\infty (1 - q^j)^n} \left.\left(\widehat{R}_\mathfrak{sp} \operatorname{ch} \mathbb{W}^\ell_\mathfrak{sp} \left( w \circ \lambda - 2( \ell + h^\vee_\mathfrak{sp} )\mu \right) \right) \right|_{\mathrm{e}^\alpha \mapsto q^{-(\alpha \mid \rho^\vee_\mathfrak{sp})}}.
\end{align*}
Hence
\begin{multline}\label{eq:proof-char-osp-L-1}
\operatorname{\widetilde{ch}} \mathbf{L}_\ell(\chi_{\lambda - 2( \ell+h^\vee_\mathfrak{sp} )\mu})
= \sum_{w \in W} (-1)^{l(w)} \frac{q^{m_\ell^\mathfrak{sp}\left( w \circ \lambda - 2(\ell+h^\vee_\mathfrak{sp}) \mu \right)}}{\prod_{j=1}^\infty (1 - q^j)^n}\\
\times \left.\left(\widehat{R}_\mathfrak{sp} \operatorname{ch} \mathbb{W}^\ell_\mathfrak{sp} \left( w \circ \lambda - 2( \ell + h^\vee_\mathfrak{sp} )\mu \right) \right) \right|_{\mathrm{e}^\alpha \mapsto q^{-(\alpha \mid \rho^\vee_\mathfrak{sp})}}
\end{multline}
Now, it is easy to see from the definitions of the Wakimoto modules that
\begin{align}\label{eq:proof-char-osp-L-2}
\operatorname{ch} \mathbb{W}^\ell_\mathfrak{sp}(\nu) = \operatorname{ch} M^\ell_\mathfrak{sp}(\nu),\quad
\nu \in \mathfrak{h}^*.
\end{align}
Using \eqref{eq:proof-char-osp-L-1}, \eqref{eq:proof-char-osp-L-2} and Lemma \ref{lemma:delta-lambda-mu}, it follows that
\begin{align*}
\operatorname{\widetilde{ch}} \mathbf{L}_\ell(\chi_{\lambda - 2( \ell+h^\vee_\mathfrak{sp} )\mu})
= &\sum_{w \in W} (-1)^{l(w)} \frac{q^{\Delta_{w \circ \lambda, \mu}}}{\prod_{j=1}^\infty (1 - q^j)^n}
\end{align*}
Comparing this with \eqref{eq:branching-1}, we see immediately that the branching functions coincide with the characters of $\mathbf{L}_\ell(\chi_{\lambda - 2( \ell+h^\vee_\mathfrak{sp} )\mu})$.
\end{proof}
\end{theorem}

\bibliographystyle{hunsrt}

\end{document}